\documentclass[12pt,reqno]{amsart}

\usepackage{amscd}
\usepackage{amsfonts}
\usepackage{amssymb}
\usepackage{amsgen}
\usepackage{amsmath}
\usepackage{amsopn}
\usepackage{verbatim}
\usepackage{xypic}
\usepackage{xspace}
\usepackage{multicol}
\usepackage{url}
\usepackage{upref}
\usepackage{pgf}
\usepackage{tikz}

\newtheorem{thm}{Theorem}[section]
\newtheorem{lem}[thm]{Lemma}
\newtheorem{prop}[thm]{Proposition}
\newtheorem{cor}[thm]{Corollary}

\newtheorem{defn}[thm]{Definition}

\newtheorem{rem}[thm]{Remark}

\newtheorem{eg}[thm]{Example}

 \def\ppmod#1{\, {\rm (mod\, }{#1)}}

 \font\cyr=wncyr10
 \newcommand{\nc}{\newcommand}

\DeclareMathOperator{\rank}{{rank}}

\nc{\per}[1]{\underset{#1}{\boldsymbol \pi}\,}
\nc{\sha}{{\mbox{\cyr x}}}

 \nc{\MT}{{\rm MT}}
 \nc{\XX}{{X}}
 \nc{\gF}{{\varPhi}}
 \nc{\ot}{\otimes}

 \nc{\wht}{\widehat}
 \nc{\bwg}{{\bigwedge}}
 \nc{\wg}{{\wedge}}
 \nc{\mmu}{{\boldsymbol{\mu}}}
 \nc{\mal}{{{\scriptstyle \maltese}}}
 \nc{\fA}{{\mathfrak A}}
 \nc{\HH}{{\mathfrak H}}
 \nc{\ra}{\rightarrow}
 \nc{\ors}{{\bfs}}
 \nc{\orr}{{\bfr}}
 \nc{\os}{{\overset}}
 \nc{\G}{{\mathbb G}}
 \nc{\F}{{\mathbb F}}
 \nc{\HHH}{{\mathbb H}}
 \nc{\Z}{{\mathbb Z}}
 \nc{\R}{{\mathbb R}}
 \nc{\N}{{\mathbb N}}
 \nc{\ZN}{{\mathbb Z_{\ge 0}}}
 \nc{\Q}{{\mathbb Q}}
 \nc{\CC}{{\mathbb C}}
 \nc{\CP}{{\mathbb{CP}}}
 \nc{\Cnn}{{\mathbb C}_{\ge 0}}
 \nc{\Cp}{{\mathbb C}_{>0}}
 \nc{\MPV}{{\mathcal{MPV}}}
 
 \nc{\tB}{{\tilde B}}
 \nc{\tg}{{\tilde{g}}}
 \nc{\tn}{{\tilde{n}}}
 \nc{\tgz}{{\tilde{\zeta}}}

 \nc{\tPsi}{{\tilde{\Psi}}}

 \nc{\od}{{\rm{od}}}
 \nc{\tz}{{\tilde{\zeta}}}
 \nc{\suf}{{\ast\,}}
 \nc{\sufq}{{\ast_q\,}}
 \nc{\gam}{{\gamma}}
 \nc{\gG}{{\Gamma}}
 \nc{\om}{{\omega}}
 \nc{\vep}{{\varepsilon}}
 \nc{\ga}{{\alpha}}
 \nc{\gl}{{\lambda}}
 \nc{\gb}{{\beta}}
 \nc{\gf}{{\varphi}}
 \nc{\gd}{{\delta}}
 \nc{\orgd}{{\vec \gd\,}}
 \nc{\gs}{{\sigma}}
 \nc{\gth}{{\theta}}
 \nc{\gS}{{\Sigma}}

 \nc{\gk}{{\kappa}}
  \nc{\gz}{{\zeta}}
 \nc{\gO}{{\Omega}}
 \nc{\sif}{{\mathcal S}}
 \nc{\gt}{{\tau}}
 \nc{\Lra}{\Longrightarrow}
 \nc{\lra}{\longrightarrow}
 \nc{\lmaps}{\longmapsto}
 \nc{\fS}{{\mathfrak S}}
 \nc{\DD}{{\mathfrak D}}
 \nc{\Llra}{\Longleftrightarrow}
 \nc{\ol}{\overline}
 \nc{\ola}{\overleftarrow}
 \nc{\lms}{\longmapsto}
 \nc{\cv}{{{\mathsf c}{\mathsf v}}}
 \nc{\zq}{{\zeta_q}}
 \nc\qup{{q\uparrow 1}}
 \nc{\us}{\underset}
 \nc{\gD}{{\Delta}}
 \nc{\bi}{{\bf i}}
 \nc{\bfone}{{\bf 1}}

 \nc{\bfa}{{\bf a}}
 \nc{\bfb}{{\bf b}}
 \nc{\bfc}{{\bf c}}
 \nc{\bfd}{{\bf d}}
 \nc{\bfe}{{\bf e}}
 \nc{\bff}{{\bf f}}
 \nc{\bfg}{{\bf g}}
 \nc{\bfi}{{\bf i}}
 \nc{\bfj}{{\bf j}}

 \nc{\bfn}{{\bf n}}
 \nc{\bfl}{{\bf l}}
 \nc{\bfk}{{\bf k}}
 \nc{\bfm}{{\bf m}}
 \nc{\bfo}{{\bf o}}
 \nc{\bfp}{{\bf p}}
 \nc{\bfq}{{\bf q}}
 \nc{\bfr}{{\bf r}}
 \nc{\bfs}{{\bf s}}
 \nc{\bft}{{\bf t}}
 \nc{\bfu}{{\bf u}}
 \nc{\bfv}{{\bf v}}
 \nc{\bfw}{{\bf w}}
 \nc{\bfx}{{\bf x}}
 \nc{\bfy}{{\bf y}}
 \nc{\bfz}{{\bf z}}
 \nc{\bfB}{{\bf B}}
 \nc{\bfP}{{\bf P}}
 \nc{\bfQ}{{\bf Q}}
 \nc{\bfY}{{\bf Y}}
 \nc{\bfgb}{{\boldsymbol \gb}}
 \nc{\bfga}{{\boldsymbol \ga}}
 \nc{\bfrho}{{\boldsymbol \rho}}
 \nc{\bfchi}{{\boldsymbol \chi}}
 \nc{\QX}{{\Q\langle \bfX\rangle}}
 \nc{\QY}{{\Q\langle \bfY\rangle}}
 \nc{\CX}{{\C\langle \bfX\rangle}}
 \nc{\CY}{{\C\langle \bfY\rangle}}
 \nc{\QXX}{{\Q\langle\!\langle \bfX\rangle\!\rangle}}
 \nc{\QYY}{{\Q\langle\!\langle \bfY\rangle\!\rangle}}
 \nc{\CXX}{{\C\langle\!\langle \bfX\rangle\!\rangle}}
 \nc{\CYY}{{\C\langle\!\langle \bfY\rangle\!\rangle}}

 \nc{\bbA}{{\mathbb A}}
 \nc{\bbB}{{\mathbb B}}
 \nc{\bbC}{{\mathbb C}}
 \nc{\bbD}{{\mathbb D}}
 \nc{\bbE}{{\mathbb E}}
 \nc{\bbF}{{\mathbb F}}
 \nc{\bbG}{{\mathbb G}}
 \nc{\bbH}{{\mathbb H}}
 \nc{\bbI}{{\mathbb I}}
 \nc{\bbJ}{{\mathbb J}}
 \nc{\bbK}{{\mathbb K}}
 \nc{\bbL}{{\mathbb L}}
 \nc{\bbM}{{\mathbb M}}
 \nc{\bbN}{{\mathbb N}}
 \nc{\bbO}{{\mathbb O}}
 \nc{\bbP}{{\mathbb P}}
 \nc{\bbQ}{{\mathbb Q}}
 \nc{\bbR}{{\mathbb R}}
 \nc{\bbS}{{\mathbb S}}
 \nc{\bbT}{{\mathbb T}}
 \nc{\bbU}{{\mathbb U}}
 \nc{\bbV}{{\mathbb V}}
 \nc{\bbW}{{\mathbb W}}
 \nc{\bbX}{{\mathbb X}}
 \nc{\bbY}{{\mathbb Y}}
 \nc{\bbZ}{{\mathbb Z}}
 \nc{\bba}{{\mathbb a}}
 \nc{\bbb}{{\mathbb b}}
 \nc{\bbc}{{\mathbb c}}
 \nc{\bbd}{{\mathbb d}}
 \nc{\bbe}{{\mathbb e}}
 \nc{\bbf}{{\mathbb f}}
 \nc{\bbg}{{\mathbb g}}
 \nc{\bbh}{{\mathbb h}}
 \nc{\bbi}{{\mathbb i}}
 \nc{\bbk}{{\mathbb k}}
 \nc{\bbl}{{\mathbb l}}
 \nc{\bbm}{{\mathbb m}}
 \nc{\bbn}{{\mathbb n}}
 \nc{\bbo}{{\mathbb o}}
 \nc{\bbp}{{\mathbb p}}
 \nc{\bbq}{{\mathbb q}}
 \nc{\bbr}{{\mathbb r}}
 \nc{\bbs}{{\mathbb s}}
 \nc{\bbt}{{\mathbb t}}
 \nc{\bbu}{{\mathbb u}}
 \nc{\bbv}{{\mathbb v}}
 \nc{\bbw}{{\mathbb w}}
 \nc{\bbx}{{\mathbb x}}
 \nc{\bby}{{\mathbb y}}
 \nc{\bbz}{{\mathbb z}}

 \nc{\MZV}{{\mathcal{MZV}}}
 \nc{\calA}{{\mathcal A}}
 \nc{\calB}{{\mathcal B}}
 \nc{\calC}{{\mathcal C}}
 \nc{\calD}{{\mathcal D}}
 \nc{\calE}{{\mathcal E}}
 \nc{\calF}{{\mathcal F}}
 \nc{\calG}{{\mathcal G}}
 \nc{\calH}{{\mathcal H}}
 \nc{\calI}{{\mathcal I}}
 \nc{\calJ}{{\mathcal J}}
 \nc{\calK}{{\mathcal K}}
 \nc{\calL}{{\mathcal L}}
 \nc{\calM}{{\mathcal M}}
 \nc{\calN}{{\mathcal N}}
 \nc{\calO}{{\mathcal O}}
 \nc{\calP}{{\mathcal P}}
 \nc{\calQ}{{\mathcal Q}}
 \nc{\calR}{{\mathcal R}}
 \nc{\calS}{{\mathcal S}}
 \nc{\calT}{{\mathcal T}}
 \nc{\calU}{{\mathcal U}}
 \nc{\calV}{{\mathcal V}}
 \nc{\calW}{{\mathcal W}}
 \nc{\calX}{{\mathcal X}}
 \nc{\calY}{{\mathcal Y}}
 \nc{\calZ}{{\mathcal Z}}
  \nc{\cala}{{\mathcal a}}
 \nc{\calb}{{\mathcal b}}
 \nc{\calc}{{\mathcal c}}
 \nc{\cald}{{\mathcal d}}
 \nc{\cale}{{\mathcal e}}
 \nc{\calf}{{\mathcal f}}
 \nc{\calg}{{\mathcal g}}
 \nc{\calh}{{\mathcal h}}
 \nc{\cali}{{\mathcal i}}
 \nc{\calj}{{\mathcal j}}
 \nc{\calk}{{\mathcal k}}
 \nc{\call}{{\mathcal l}}
 \nc{\calm}{{\mathcal m}}
 \nc{\caln}{{\mathcal n}}
 \nc{\calo}{{\mathcal o}}
 \nc{\calp}{{\mathsf p}}
 \nc{\calq}{{\mathcal q}}
 \nc{\calr}{{\mathcal r}}
 \nc{\cals}{{\mathcal s}}
 \nc{\calt}{{\mathcal t}}
 \nc{\calu}{{\mathcal u}}
 \nc{\calv}{{\mathcal v}}
 \nc{\calw}{{\mathcal w}}
 \nc{\calx}{{\mathcal x}}
 \nc{\caly}{{\mathcal y}}
 \nc{\calz}{{\mathcal z}}

 \nc{\frakA}{{\mathfrak A}}
 \nc{\frakB}{{\mathfrak B}}
 \nc{\frakC}{{\mathfrak C}}
 \nc{\frakD}{{\mathfrak D}}
 \nc{\frakE}{{\mathfrak E}}
 \nc{\frakF}{{\mathfrak F}}
 \nc{\frakG}{{\mathfrak G}}
 \nc{\frakH}{{\mathfrak H}}
 \nc{\frakI}{{\mathfrak I}}
 \nc{\frakJ}{{\mathfrak J}}
 \nc{\frakK}{{\mathfrak K}}
 \nc{\frakL}{{\mathfrak L}}
 \nc{\frakM}{{\mathfrak M}}
 \nc{\frakN}{{\mathfrak N}}
 \nc{\frakO}{{\mathfrak O}}
 \nc{\frakP}{{\mathfrak P}}
 \nc{\frakQ}{{\mathfrak Q}}
 \nc{\frakR}{{\mathfrak R}}
 \nc{\frakS}{{\mathfrak S}}
 \nc{\frakT}{{\mathfrak T}}
 \nc{\frakU}{{\mathfrak U}}
 \nc{\frakV}{{\mathfrak V}}
 \nc{\frakW}{{\mathfrak W}}
 \nc{\frakX}{{\mathfrak X}}
 \nc{\frakY}{{\mathfrak Y}}
 \nc{\frakZ}{{\mathfrak Z}}
 \nc{\fraka}{{\mathfrak a}}
 \nc{\frakb}{{\mathfrak b}}
 \nc{\frakc}{{\mathfrak c}}
 \nc{\frakd}{{\mathfrak d}}
 \nc{\frake}{{\mathfrak e}}
 \nc{\frakf}{{\mathfrak f}}
 \nc{\frakg}{{\mathfrak g}}
 \nc{\frakh}{{\mathfrak h}}
 \nc{\fraki}{{\mathfrak i}}
 \nc{\frakj}{{\mathfrak j}}
 \nc{\frakk}{{\mathfrak k}}
 \nc{\frakl}{{\mathfrak l}}
 \nc{\frakm}{{\mathfrak m}}
 \nc{\frakn}{{\mathfrak n}}
 \nc{\frako}{{\mathfrak o}}
 \nc{\frakp}{{\mathfrak p}}
 \nc{\frakq}{{\mathfrak q}}
 \nc{\frakr}{{\mathfrak r}}
 \nc{\fraks}{{\mathfrak s}}
 \nc{\frakt}{{\mathfrak t}}
 \nc{\fraku}{{\mathfrak u}}
 \nc{\frakv}{{\mathfrak v}}
 \nc{\frakw}{{\mathfrak w}}
 \nc{\frakx}{{\mathfrak x}}
 \nc{\fraky}{{\mathfrak y}}
 \nc{\frakz}{{\mathfrak z}}
 \nc{\barfrakz}{{\bar{\mathfrak z}}}
 \nc{\barfrakI}{{\bar{\mathfrak I}}}
 \nc{\barI}{{\bar{I}}}
 \nc{\bargb}{{\bar{\gb}}}
 \nc{\tG}{{\tilde{G}}}
 \nc{\so}{{\mathfrak so}}
 \nc{\sa}{{\mbox{{\scriptsize \cyr x}}}}
 \nc{\slfour}{{\mathfrak sl}_4}
 \nc{\one}{{\bf 1}}
 \nc{\zero}{{\bf 0}}
 \nc{\Qxy}{\Q\langle x,y\rangle}
 \nc{\iij}{{i}}
\nc{\bga}{{\boldsymbol \ga}}
\nc{\gL}{{\Lambda}}

 \nc{\DZ}{\mathcal{DZ}}
 \nc{\PDZ}{\mathcal{PDZ}}
\newcommand{\ZZ}{{\mathcal Z}}

\begin{document}

\title[Double Zeta Values and Double Eisenstein Series]{Double Shuffle Relations of Double Zeta Values \\
and Double Eisenstein Series at Level $N$}


\begin{abstract}
In their seminal paper ``Double zeta values and modular forms''
Gangl, Kaneko and Zagier defined a double Eisenstein series and used it to study
the relations between double zeta values. One of their key ideas is to study the formal
double space and apply the double shuffle relations. They also proved the double
shuffle relations for the double Eisenstein series.
More recently, Kaneko and Tasaka extended the double Eisenstein series to level 2,
proved its double shuffle relations and studied the double zeta values at level 2.
Motivated by the above works, we define in this paper
the corresponding objects at higher levels and prove that the double Eisenstein series at level
$N$ satisfies the double shuffle relations for every positive integer $N$. In order to
obtain our main theorem we prove a key result on the multiple divisor functions at level $N$
and then use it to solve a complicated under-determined system of linear equations by
some standard techniques from linear algebra.
\end{abstract}

\maketitle

\tableofcontents

\allowdisplaybreaks

\section{Introduction}
Eisenstein series have played important roles in the study of modular
forms and elliptic curves. One of their most important properties is
that the constant term of their Fourier series expansion
is essentially given by the Riemann zeta values at even weight.
In the seminal paper \cite{GKZ2006} Gangl, Kaneko
and Zagier defined a double Eisenstein series and used it to
study the relations between double zeta values.
One of their key ideas is to study the formal double space
and apply the double shuffle relations. They then proved the double
shuffle relations for the double Eisenstein series. The double zeta
relations have also been considered by Baumard and Schneps in
\cite{BaumardSc2011} from the point of view of period polynomials
and double shuffle Lie algebra defined by Ihara. More recently,
Kaneko and Tasaka  \cite{KanekoTa2013} and
Nakamura and Tasaka \cite{NakamuraTa2013} extended the double
Eisenstein series to level 2, proved its double shuffle relations
and studied the double zeta values at level 2.
Motivated by the above works, we define in this paper
the corresponding objects at higher levels and consider the
double shuffle relations satisfied by them.

We follow the notation in \cite{GKZ2006}: for any $m,c\in\Z$
and $\tau\in \HHH$ (upper half plane), we write
$m\tau +c \succ 0$ if $m>0$ or $m=0$ and $c>0$ and
$m\tau +c \succ n\tau +d$ if $m\tau+c-(n\tau+d) \succ 0$.
For any $\bfa=(a_1,\dots,a_d)\in(\Z/N\Z)^d$
and $\bfs=(s_1,\dots,s_d)\in\N^d$ we define the multiple
Eisenstein series at level $N$ by
\begin{equation}\label{defn:Eserie}
G_\bfs^\bfa(\tau)=G_\bfs^{\bfa;N}(\tau)=\sum_{\substack{
m_1 N\tau+c_1\succ \dots\succ m_d N\tau+c_d\succ 0  \\
m_j,c_j\in \Z,\
c_j \equiv a_j \ppmod{N}\ \forall j }}
\frac{1}{(m_1 N\tau+c_1)^{s_1}\cdots (m_d N\tau+c_d)^{s_d}}.
\end{equation}
Here, by convention, we often choose $0\le a<N$ to represent the
residue class congruent to $a$ modulo $N$.
It is not too hard to show that the series converges absolutely when
$s_1\ge 3$ and $s_j\ge 2$ for all $j\ge 2$.
We will call $d$ the depth and the sum $s_1+\dots+s_d$ the weight.
At level one case, Gangl, Keneko and Zagier \cite{GKZ2006} studied
the double Eisenstein series and related them to modular form
by using the Eichler-Shimura correspondence. In \cite{Bachmann2012}
Bachmann generalized this to arbitrary depth and obtained many
interesting relations among these and the classical Eisenstein series
(and the cusp form $\gD$) using the double shuffle relations.

The main idea to study the multiple Eisenstein series is by using
their Fourier series expansions with the help of the so called
multiple divisor functions at level $N$ defined as follows:
For $\bfa=(a_1,\dots,a_d)\in(\Z/N\Z)^d$ and
$\bfs=(s_1,\dots,s_d)\in(\N\cup\{0\})^d$
\begin{equation}\label{equ:Mdivisor}
\gs_\bfs^\bfa(m)=\gs_\bfs^{\bfa;N}(m)=\sum_{\substack{
u_1v_1+\dots+u_dv_d=m \\
u_1>\dots>u_d>0 \\
u_j,v_j\in \N\ \forall j=1,\dots,d  }}
\eta^{a_1v_1+\dots+a_d v_d} v_1^{s_1} \dots v_d^{s_d}
\end{equation}
where $\eta=\eta_N=\exp(2\pi i/N)$ is the primitive $N$th
root of unity and $v_1, \cdots, v_d$ are positive integers.
Obviously, one can recover
the classical divisor function by setting $N=d=1$.

We now briefly describe the content of the paper.
In the next section we shall first define the double zeta
values at level $N$ and write down explicitly the double
shuffle relations satisfied by these values. Then
we consider the same problem in the formal vector space
corresponding to the double zeta values. As consequences of
these double shuffle relations we prove two sum formulas
in Theorem~\ref{thm:sumformula}.

By the general philosophy, the constant terms of level $N$
multiple Eisenstein series are
given by the level $N$ multiple zeta values. Further, we expect that
level $N$ multiple zeta values satisfy the double shuffle relations such as
those given in Proposition~\ref{prop:dbsfreg}. Hence we would like to
know if the corresponding level $N$ multiple Eisenstein series also
satisfy similar relations. When $N=1$ this has been studied by
Bachmann and Tasaka \cite{BachmannTa2015}.

The main goal of the paper is to prove
Theorem~\ref{thm:EisenSeriesDSforG} which gives the double
shuffle relations of double Eisenstein series at level $N$
for every positive integer $N$. The difficulty in
generalizing the known $N=1$ and $N=2$ cases to arbitrary
levels lies in the fact that there are many choices of the
constant terms in the generating function of the double
Eisenstein series and the other related series. It turns out
that this fact is a consequence of an under-determined system
of linear equations with $(3N^2+N)/2$ variables and $N^2+N$
equations. Essentially, we need to show these equations
are consistent with each other. For this we need a key result
concerning $N$-th roots of unity and the multiple divisor
functions at level $N$ which will be proved in section
\ref{sec: multipleDivisorFunction}. In the last section,
using some standard techniques from linear algebra we prove the
solvability of the linear system mentioned above. This
enables us to derive our main result
Theorem~\ref{thm:EisenSeriesDS} which generalizes
\cite[Theorem~7]{GKZ2006} and \cite[Theorem~3]{KanekoTa2013}

\section{The multiple zeta values at level $N$}
For any $\bfs=(s_1,\dots,s_d)\in\N^d$ with $s_1\ge 2$ and
$\bfa=(a_1,\dots,a_d)\in\Z/N\Z$,  we define the
\emph{multiple zeta values at level} $N$ by
\begin{equation}\label{equ:zetavalLevelN}
\zeta_N^\bfa(\bfs) = \sum_{n_1>\cdots>n_d>0,\ n_j\equiv a_j\ppmod{N}\ \forall j}
\frac{1}{n_1^{s_1}\cdots n_d^{s_d}}.
\end{equation}
These numbers are rational multiples of the multiple Hurwitz zeta values
which have been studied by many authors. If we define the multiple polylogarithms
\begin{equation*}
Li_\bfs(x_1,\dots,x_d)= \sum_{n_1>\cdots>n_d>0}
\frac{x_1^{n_1}\dots x_d^{n_d}}{n_1^{s_1}\cdots n_d^{s_d}}.
\end{equation*}
Then it is not hard to see that
\begin{equation}\label{equ:zetavalLevelNrelLi}
\zeta_N^\bfa(\bfs)=  \frac1{N^d} \sum_{\gb_1=1}^N\cdots\sum_{\gb_d=1}^N
    \eta^{-\gb_1 a_1-\dots-\gb_d a_d} Li_\bfs(\eta^{\gb_1},\dots,\eta^{\gb_d}).
\end{equation}

We now restrict ourself to levels at one or two.
We remark that our definition of the double zeta value at level two is slightly
different from that of \cite{KanekoTa2013} since we allow
$a_1=a_2=0$ in which case we in fact essentially recover the usual
double zeta values (at level 1).

We can also use Chen's iterated integrals to derive formulas similar to
\eqref{equ:zetavalLevelNrelLi} which will be
useful in the regularization of these values. Let $\om=\frac{dt}{t}$,
$\om^a_\ga=\om(N)^a_\ga=\frac{(\eta^\ga t)^{a-1}\, dt}{1-\eta^\ga t}$
($1\le \ga\le N$) be 1-forms. By the partial fraction expansion
\begin{equation*}
    \frac{t^{a-1}}{1-t^N}=\frac1N \sum_{\ga=1}^N \frac{\eta^{-\ga(a-1)}}{1-\eta^\ga t}
\end{equation*}
we see immediately that
\begin{align}\label{equ:zetavalInt}
 \zeta_N^{a} (r)=&\int_0^1 \om^{r-1} \frac{t^{a-1}\, dt}{1-t^N}
 \left(\overset{\text{def}}{=} {\mathop {\int \cdots \int}_{1>t_1>t_2>\cdots>t_{r}>0} } \frac{d  t_1}{t_1} \cdot \frac{d  t_2}{t_2} \cdots  \frac{d  t_{r-1}}{t_{r-1}} \cdot \frac{t_r^{a-1} d  t_r}{1-t_r^N}\right)\\
 =&\frac1N\sum_{\ga=1}^N \int_0^1 \om^{r-1}\frac{\eta^{-\ga(a-1)}\, dt}{1-\eta^\ga t}
 =\frac1N\sum_{\ga=1}^N \eta^{-\ga a} Li_r(\eta^\ga).    \notag
\end{align}\noindent
Furthermore,
\begin{align}
\zeta_N^{a,b}  (r,s)=&\int_0^1 \om^{r-1} \frac{t^{a-b-1}\, dt}{1-t^N}
        \om^{s-1} \frac{t^{b-1}\, dt}{1-t^N} \notag \\
=&\frac1{N^2}\sum_{\ga=1}^N \sum_{\gb=1}^N \int_0^1 \om^{r-1}\frac{\eta^{-\ga(a-b-1)}\, dt}{1-\eta^\ga t}
      \om^{s-1}\frac{\eta^{-\gb(b-1)}\, dt}{1-\eta^\gb t}    \label{equ:dblzetavalInt} \\
=&\frac1{N^2}\sum_{\ga=1}^N \sum_{\gb=1}^N \eta^{-\ga(a-b)-\gb b} Li_{r,s}(\eta^\ga,\eta^{\gb-\ga}) \notag
\end{align}

Our first result is the following explicit form of double shuffle relations.

\begin{prop} \label{prop:dbsf}
For positive integers $r, s\ge2$ and integers $a,b\in\Z/N\Z$, we have
\begin{align} \label{equ:stuffle}
\zeta_N^{a} (r) \zeta_N^{b} (s) = &\zeta_N^{a,b} (r,s) + \zeta_N^{b,a} (s,r) +\gd_{a,b} \zeta_N^{a} (s+r)\\
= & \sum_{i=0}^{s-1}\binom{r+i-1}{r-1} \zeta_N^{a+b,b} (r+i,s-i)
    +  \sum_{j=0}^{r-1}\binom{s+j-1}{s-1} \zeta_N^{a+b,a} (s+j,r-j) \notag \\
=&\sum_{\substack{i+j=r+s\\  i\ge 2, j\ge 1}} \left(\binom{i-1}{r-1} \zeta_N^{a+b,b}(i,j)
    + \binom{i-1}{s-1} \zeta_N^{a+b,a}(i,j)\right), \notag
\end{align}\noindent
where $\gd_{a,b}$ is the Kronecker symbol, namely, $\gd_{a,b}=1$ if $a=b$ and $\gd_{a,b}=0$ if $a\ne b$.
\end{prop}

\begin{proof} 
The first equality is clear by the definition \eqref{equ:zetavalLevelN}.
The second equality follows immediately from the shuffle product
formula of iterated integrals \cite[(1.5.1)]{Chen1971}:
\begin{equation*}
 \int_0^1 \om_1\cdots \om_r\int_0^1 \om_{r+1}\cdots \om_{r+s}
=\sum_\gs \int_0^1 \om_{\gs(1)}\cdots \om_{\gs(r+s)},
\end{equation*}
where $\gs$ ranges over all shuffles of type $(r,s)$, i.e.,
permutations $\gs$ of $r+s$ letters with
$\gs^{-1}(1)<\cdots<\gs^{-1}(r)$
and $\gs^{-1}(r+1)<\cdots< \gs^{-1}(r+s)$.
\end{proof}

\section{Double zeta space at level $N$}
Now we introduce the level $N$ version of the \emph{formal double zeta space} studied in \cite{GKZ2006} as follows.
Let $k\geq 2$ and $\DZ(N)_k$ be the $\Q$-vector space spanned by the formal symbols
$Z_{r,s}^{a,b}=Z(N)_{r,s}^{a,b}$, $P_{r,s}^{a,b}=P(N)_{r,s}^{a,b}$, and $Z^{a}_k=Z(N)^{a}_k$
$(r,s\ge1, r+s=k, a,b\in\Z/N\Z)$ with the set of relations
\begin{equation} \label{equ:PNab}
 P_{r,s}^{a,b} = Z_{r,s}^{a,b} + Z_{s,r}^{b,a} + \gd_{a,b} Z^a_{s+r}
 = \sum_{\substack{i+j=k\\
i,j\ge 1}} \left(\binom{i-1}{r-1} Z^{a+b,b}_{i,j} + \binom{i-1}{s-1} Z^{a+b,a}_{i,j}\right)
\end{equation}
for $r,s\ge1,\, r+s=k$. Namely,
\begin{equation*}
 \DZ(N)_k = \frac{\Q\langle Z(N)_{r,s}^{a,b},P(N)_{r,s}^{a,b}, Z(N)^{a}_k: \  a,b\in\Z/N\Z,\ r,s\ge1,\, r+s=k\rangle  }{\Q\langle \mbox{relations\ } \eqref{equ:PNab} \rangle }.
\end{equation*}
Clearly
\begin{equation*}
 \DZ(N)_k = \frac{\Q\langle Z(N)_{r,s}^{a,b}, Z(N)^{a}_k: \  a,b\in\Z/N\Z,\ r,s\ge1,\, r+s=k \rangle  }{\Q\langle \mbox{relations\ } \eqref{equ:ZNshuffle} \rangle },
\end{equation*}
where the defining relations are (dropping the dependence on $N$)
\begin{equation} \label{equ:ZNshuffle}
Z_{r,s}^{a,b} + Z_{s,r}^{b,a}+ \gd_{a,b} Z^a_{k} = \sum_{\substack{i+j=k\\
i,j\ge 1}} \left(\binom{i-1}{r-1} Z^{a+b,b}_{i,j} + \binom{i-1}{s-1} Z^{a+b,a}_{i,j}\right).
\end{equation}

Recall that we often choose $0\le a<N$ to represent the residue class congruent to $a$ modulo $N$. Observe that when residue class $\bar 0$ appears in any conditions involving gcd we should use $N$ to represent it. For example, $\gcd(0,0)=\gcd(0,N)=\gcd(N,N)=N$.
We now may define $\PDZ(N)_k$, the formal double zeta space of 
\emph{pure level} $N$ by restricting $Z_{s,r}^{b,a}$ to
$$\gO(N)=\{(a,b): 0\le a,b< N, \gcd(a,b,N)=1\}$$
and $Z^{a}_{k}$ to $\{a: 1\le a< N,\gcd(a,N)=1\}$ in the above. This is well-defined since if $\gcd(a,b,N)=1$ then $\gcd(a+b,a,N)=\gcd(a+b,b,N)=1$.

Since both sides of \eqref{equ:ZNshuffle} are invariant under $(a,b;r,s)\leftrightarrow (b,a;s,r)$
we may just take $r\le s$. Thus for even $k$ the group $\DZ(N)_k$ has $(k-1)N^2+N$ generators and $k N^2/2$ relations. Hence
\begin{equation*}
   \dim \DZ(N)_k \ge \frac{(k-2)N^2+2N}{2}.
\end{equation*}
Similarly, for even $k$ the group $\PDZ(N)_k$ has $(k-1)|\gO(N)|+\gf(N)$ generators and $(k-1)(|\gO(N)|+\gf(N))/2$ relations. Hence
\begin{equation}\label{equ:PDZ(N)bound}
   \dim \PDZ(N)_k \ge \frac{(k-1)(|\gO(N)|-\gf(N))}{2}+\gf(N).
\end{equation}
\begin{rem}
(a). Notice that the double zeta space $\DZ(2)_k$ in \cite{KanekoTa2013}
is our $\PDZ(2)_k$. (b). The bound in \eqref{equ:PDZ(N)bound} is not sharp. For example,
when $(N,k)=(3,4)$ we have 24 generators and only 13 independent relations instead of 15. So
$\dim \PDZ(3)_4=11>24-15.$
\end{rem}

Note that the relations \eqref{equ:PNab} (as well as \eqref{equ:ZNshuffle}) correspond to those in Proposition~\ref{prop:dbsf} when $r,s\ge2$,
under the correspondences
\begin{align*}
&Z(N)_{r,s}^{a,b}\longleftrightarrow \zeta_N^{a,b}(r,s),
\ \ Z(N)^{a}_k\longleftrightarrow \zeta_N^{a}(k), \ \
  P(N)_{r,s}^{a,b}\longleftrightarrow \zeta_N^{a}(r)\zeta_N^{b}(s),
\end{align*}\noindent
the binomial coefficients for $i=1$ on the right vanish in both \eqref{equ:PNab} and \eqref{equ:ZNshuffle}.  For our later applications
it is convenient to allow the ``divergent'' $Z(N)^{a,b}_{1,k-1}$ and $P(N)^{a,b}_{1,k-1}$ etc., and in fact
the double shuffle relations in Proposition~\ref{prop:dbsf} can be extended for $r=1$ or $s=1$ by using a suitable regularization procedure for $Li^{\sha}_{1,k-1}(1,\eta)$ etc.\ developed in \cite{ArakawaKa1999} which was motivated by \cite{IKZ2006}. For a comprehensive treatment of the general multiple zeta values of level $N$,
please see our paper \cite{YuanZh2014b}.
Specifically, in our current situation we can define the following renormalized values.
Let $T$ be a formal variable,
\begin{itemize}
  \item Note that $Li_1^\ast (1)=\zeta_*(1)=T$
and $Li_1^\sha(1)=\zeta_\sha(1)=T$. By \eqref{equ:zetavalLevelNrelLi} and
\eqref{equ:zetavalInt}, for $a\in\Z/N\Z$
 \begin{equation}\label{equ:ZN1regval}
\zeta_{N;\ast}^a(1)=\zeta_{N;\sha}^a(1) =
\frac{1}{N} \left(T+\sum_{\ga=1}^{N-1}  \eta^{-a\ga} Li_1(\eta^\ga)
\right).
\end{equation}

  \item  By \eqref{equ:stuffle}, for $s\ge2$ and $a,b\in\Z/N\Z$
\begin{align*}
\zeta_{N;\ast}^{a,a}(1,s)&=\frac{1}{N}\left(T+\sum_{\ga=1}^{N-1}  \eta^{-a\ga} Li_1(\eta^\ga)
\right)\zeta_N^a(s) -\zeta_N^{a,a}(s,1)-\zeta_N^{a}(s+1) , \\
\zeta_{N;\ast}^{a,b}(1,s)&=\frac{1}{N}\left(T+\sum_{\ga=1}^{N-1}  \eta^{-a\ga} Li_1(\eta^\ga)
\right)\zeta_N^b(s) -\zeta_N^{b,a}(s,1) \text{ if } \quad a\ne b,
\end{align*}

  \item By \eqref{equ:zetavalLevelNrelLi}, for $a,b\in\Z/N\Z$
\begin{align*}
\zeta_{N;\ast}^{a,b}(1,1)=& \frac{1}{N^2}\left(\frac{T^2}2
+\sum_{\gb=1}^{N-1}  \eta^{-b\gb}\big(TLi_{1}(\eta^\gb)-Li_{1,1}(\eta^\gb,\eta^{-\gb})\big) \right.\\  &\hskip3cm \left.+\sum_{\ga=1}^{N-1}\sum_{\gb=1}^{N}  \eta^{-a\ga-b\gb} Li_{1,1}(\eta^\ga,\eta^{\gb-\ga}) \right).
\end{align*}

  \item By \eqref{equ:dblzetavalInt}, for $a,b\in\Z/N\Z$
\begin{equation*}
\zeta_{N;\sha}^{a,b}(1,s)=\frac1{N^2}\sum_{\ga=1}^N \sum_{\gb=1}^N \eta^{-\ga(a-b)-\gb b} Li^{\sha}_{1,s}(\eta^\ga,\eta^{\gb-\ga}) \notag
\end{equation*}
where $Li^\sha_{1,1}(1,1)=\frac12 T^2$,
$Li^\sha_{1,s}(\eta^\ga,\eta^{\gb-\ga}) = Li_{1,s}(\eta^\ga,\eta^{\gb-\ga})$ for all $\ga\ne \bar{0}\in\Z/N\Z$, and
\begin{equation*}
 Li^\sha_{1,s}(1,\eta^{\gb}) =TLi_{s}(\eta^{\gb})-\sum_{t=2}^s Li_{t,s+1-t}(1,\eta^{\gb})-Li_{s,1}(\eta^\gb,\eta^{-\gb}) \quad \forall (s,\gb)\ne (1,\bar{0}).
\end{equation*}
\end{itemize}
The equations in Proposition~\ref{prop:dbsf} are valid for all $r,\,s\ge1$. Here we have used the fact that for
$\ga,\gb\in\Z/N\Z$ ($\ga\ne \bar{0}$) we have
\begin{equation}\label{equ:LiEval}
Li_1(\eta^\ga) = \int_0^1 \frac{\eta^\ga\, dt}{1-\eta^\ga t},\quad Li_{1,1}(\eta^\ga,\eta^{\gb-\ga}) = \int_{1>t_1>t_2>0} \frac{\eta^\ga\, dt_1}{1-\eta^\ga t_1}\frac{\eta^\gb\, dt_2}{1-\eta^\gb t_2}.
\end{equation}

With the above regularized values we can now extend Proposition~\ref{prop:dbsf}
to these cases. For $r\ge 2$ and $s\ge 1$ we set
$\zeta_{N;\ast}^{a} (r)=\zeta_{N;\sha}^{a} (r)=\zeta_N^{a} (r)$
and $\zeta_{N;\ast}^{a,b} (r,s)=\zeta_{N;\sha}^{a,b} (r,s)=\zeta_N^{a,b} (r,s)$.

\begin{prop} \label{prop:dbsfreg}
For positive integers $r, s\ge 1$ and $a,b\in\Z/N\Z$, we have
\begin{align*}
\zeta_{N;\ast}^{a} (r) \zeta_{N;\ast}^{b} (s) = &\zeta_{N;\ast}^{a,b} (r,s) + \zeta_{N;\ast}^{b,a} (s,r) +\gd_{a,b} \zeta_{N;\ast}^{a} (s+r)\\
\zeta_{N;\sha}^{a} (r) \zeta_{N;\sha}^{b} (s)
=&\sum_{\substack{i+j=r+s\\  i, j\ge 1}} \left(\binom{i-1}{r-1} \zeta_{N;\sha}^{a+b,b}(i,j)
    + \binom{i-1}{s-1} \zeta_{N;\sha}^{a+b,a}(i,j)\right),
\end{align*}\noindent
where we set $\binom{0}{0}=1$.
\end{prop}
\begin{proof}
We only need to check the relations in the case when $r=1,s\ge 2$ or
$r\ge 2, s=1$ or $r=s=1$.
These follows directly from the definitions and the stuffle and
shuffle relations among $Li_1$ and $Li_{1,1}$. We leave the
details to the interested reader.
\end{proof}

The following theorem generalizes both a result of \cite{KanekoTa2013}
and a result of \cite{GKZ2006}.
\begin{thm} \label{thm:sumformula}
Let $k$ be a positive even integer and $a\in\Z/N\Z$. Then
\begin{align}
\sum_{1\le r<k,\ r \text{ odd}} Z^{a,a}_{r,k-r}=\frac14\Big(
2Z^{N,a}_{1,k-1}+2Z^{2a,a}_{1,k-1}-Z^a_{k}+2\gd_{a,\bar{0}}Z^N_{k}\Big), \label{equ:sumformulaOdd} \\
\sum_{1< r<k,\ r \text{ even}} Z^{a,a}_{r,k-r}=\frac14\Big(
2Z^{N,a}_{1,k-1}-2Z^{2a,a}_{1,k-1}+Z^a_{k}+2\gd_{a,\bar{0}}Z^N_{k}\Big).  \label{equ:sumformulaEven}
\end{align}
\end{thm}

\begin{proof}
Consider the generating functions
\begin{equation*}
\ZZ^{a,b}_k (X,Y)  = \sum_{r+s=k} Z_{r,s}^{a,b} X^{r-1} Y^{s-1}.
\end{equation*}
By \eqref{equ:ZNshuffle} we see that
\begin{equation*}
\ZZ^{a,b}_k (X,Y)  + \ZZ^{b,a}_k (Y,X) + \gd_{a,b} Z^a_{k}\frac{X^{k-1}-Y^{k-1}}{X-Y} =
\ZZ^{a+b,b}_k (X+Y,Y)  + \ZZ^{a+b,a}_k (X+Y,X) .
\end{equation*}
Set $(X,Y)=(1,0)$ and then $(X,Y)=(1,-1)$ we get, respectively,
\begin{align}
Z^{a,b}_{k-1,1}+Z^{b,a}_{1,k-1}+ \gd_{a,b} Z^a_{k} =& Z^{a+b,b}_{k-1,1}+\sum_{r=1}^{k-1} Z^{a+b,a}_{r,k-r}, \label{equ:X=1,Y=0}\\
\sum_{r=1}^{k-1} (-1)^{r-1}\big( Z^{a,b}_{r,k-r}+Z^{b,a}_{r,k-r}\big)+ \gd_{a,b} Z^a_{k} =& Z^{a+b,b}_{1,k-1}+Z^{a+b,a}_{1,k-1}. \label{equ:X=1,Y=-1}
\end{align}\noindent
Setting $b=N$  in \eqref{equ:X=1,Y=0} and $a=b$ in \eqref{equ:X=1,Y=-1} we get
\begin{align}
 \sum_{r=1}^{k-1} Z^{a,a}_{r,k-r}=&  Z^{N,a}_{1,k-1}+\gd_{a,\bar{0}} Z^N_k, \label{equ:X=1,Y=02}\\
2\sum_{r=1}^{k-1} (-1)^{r-1} Z^{a,a}_{r,k-r}+Z^a_{k} =& 2Z^{2a,a}_{1,k-1}. \label{equ:X=1,Y=-12}
\end{align}\noindent
By adding (resp.\ subtracting) twice of \eqref{equ:X=1,Y=02} to
(resp.\ from) \eqref{equ:X=1,Y=-12} we obtain \eqref{equ:sumformulaOdd}
and \eqref{equ:sumformulaEven}.
\end{proof}
\begin{rem} Part 1) of \cite[Theorem 1]{KanekoTa2013} follows from the special case
of $N=2$ and $a=1$ of our theorem. By taking $N=a=1$ in the theorem we obtain \cite[Theorem 1]{GKZ2006}.
\end{rem}

Next we describe the linear relations among $Z^{a,b}_{i,j}$'s
using some homogeneous polynomials.
\begin{prop}\label{prop:polynomialRelation}
Let $k\ge 2$ be a positive integer. Let $c_{i,j}^{a,b}\in\Q$ for all $i,j\in\N$ and $a,b\in\Z/N\Z$. Then
the following two statements are equivalent:
\begin{enumerate}
  \item[\upshape{(i)}] The relation
\begin{equation*}
    \sum_{0\le a\le b<N} \sum_{i+j=k} c_{i,j}^{a,b} Z^{a,b}_{i,j} \equiv 0
        \pmod{\Q \langle Z^a_{k}: a\in\Z/N\Z \rangle}
\end{equation*}
holds in $\DZ(N)_k$. Here and in the rest of the paper, $\sum_{i+j=k}$ means $\sum_{i+j=k,i,j\ge 1}$.

\item[\upshape{(ii)}] There exist some homogeneous polynomials $F_{a,b}\in\Q[X,Y]$ ($0\le a\le b<N$)
of degree $k-2$ such that
\begin{multline}\label{equ:F=cPureCase}
    \sum_{0\le a\le b<N} F_{a,b}(X_b,Y_a)+F_{a,b}(Y_b,X_a)-F_{a,b}(X_{a+b}+Y_b,X_{a+b})\\
    -F_{a,b}(X_{a+b},X_{a+b}+Y_a)= \sum_{a,b\in\Z/N\Z} \sum_{i+j=k} \binom{k-2}{i-1} c_{i,j}^{a,b} X_a^{i-1} Y_b^{j-1}.
\end{multline}
\end{enumerate}
Further, the following two statements are equivalent:
\begin{enumerate}
  \item[\upshape{(iii)}] The relation
\begin{equation*}
    \sum_{a,b\in\gO(N), a\le b} \sum_{i+j=k} c_{i,j}^{a,b} Z^{a,b}_{i,j} \equiv 0
        \pmod{\Q \langle Z^a_{k}: 1\le a<N, \gcd(a,N)=1 \rangle}
\end{equation*}
holds in $\PDZ(N)_k$.

\item[\upshape{(iv)}] There exist some homogeneous polynomials $F_{a,b}\in\Q[X,Y]$ ($a,b\in\gO(N)$ and $a\le b$)
of degree $k-2$ such that
\begin{multline}\label{equ:F=c}
    \sum_{a,b\in\gO(N), a\le b} F_{a,b}(X_b,Y_a)+F_{a,b}(Y_b,X_a)-F_{a,b}(X_{a+b}+Y_b,X_{a+b})\\
    -F_{a,b}(X_{a+b},X_{a+b}+Y_a)= \sum_{a,b\in\gO(N)} \sum_{i+j=k} \binom{k-2}{i-1} c_{i,j}^{a,b} X_a^{i-1} Y_b^{j-1}.
\end{multline}
\end{enumerate}
\end{prop}
\begin{proof} For any fixed $a,b\in \Z/N\Z$ we take $F_{a,b}(X,Y)=\binom{k-2}{r-1}X^{r-1}Y^{s-1}$ ($r+s=k$)
and $F_{c,d}(X,Y)=0$ for all $(c,d)\ne(a,b)$. Then the expansion of the left hand side of  \eqref{equ:F=c}
determines the values $c_{i,j}^{a,b}$ uniquely such that \eqref{equ:ZNshuffle} holds which implies (i).
In fact, when $a\ne b$
we obtain an exact equation in (i). Since any relation of the form in (i) in $\DZ(N)_k$ should come from a linear combination of \eqref{equ:ZNshuffle} with various choices of $(a,b)\in (\Z/N\Z)^2$ modulo $\Q \langle Z^a_{k}: a\in\Z/N\Z \rangle$ and any homogeneous polynomial is a linear combination of monomials of the form $F_{a,b}(X,Y)$,
the equivalence of (i) and (ii) follows immediately.
Similar arguments clearly shows the equivalence of (iii) and (iv).
\end{proof}

\begin{rem}
Proposition~\ref{prop:polynomialRelation} generalizes \cite[Proposition~2.2(i)(ii)]{GKZ2006} and \cite[Lemma 1]{KanekoTa2013}. In fact, when $N=2$ we can take $(a,b)=(0,1),(1,1)$ in (iii) and (iv) of Proposition~\ref{prop:polynomialRelation} then we see that
$F=F_{0,1}$ and $G=F_{1,1}$ in \cite[Proposition~2.2(ii)]{GKZ2006} with relabeling of $X$'s and $Y$'s as follows: $X_0^i Y_1^j\to X_1^i Y_1^j$, $X_1^i Y_0^j\to X_2^i Y_2^j$, and
$X_1^i Y_1^j\to X_3^i Y_3^j$.
\end{rem}

\section{Fourier series expansion of the double Eisenstein series at level $N$}
In this section we will describe a procedure to find the Fourier series expansion
of double Eisenstein series. This can be generalized to larger depths.
Similar to the notation used in \cite{Bachmann2012} and \cite{GKZ2006},
for any $a\in \Z/N\Z$ and positive integer $s$ set
\begin{align*}
\Psi^a_s(\tau)=&\, \Psi^{a;N}_s(\tau) = \sum_{c\equiv a\ppmod{N},
    c\in \Z}\frac{1}{(\tau+ c)^{s}} \quad \forall s\ge 2,\\
\Psi^a_1(\tau)=&\, \Psi^{a;N}_1(\tau) = \lim_{M\to\infty}
    \sum_{c\equiv a\ppmod{N}, |c|<M}\frac{1}{\tau+ c}.
\end{align*}\noindent
Then we have
\begin{lem}\label{lem:Psi}
Let $q=e^{2\pi i \tau}$ and $\eta=\eta_N=\exp(2\pi i/N)$.
Then for any $a\in \Z/N\Z$ and $s\in \N$, we have
\begin{equation}\label{equ:Psi}
\Psi^a_s(N\tau) =
\left\{
  \begin{array}{ll}
{\displaystyle -\frac{\pi i}{N}-\frac{2\pi i}{N}\sum_{n\ge 1} \eta^{a n} q^n}, & \hbox{if $s=1$;} \\
{\displaystyle \frac{(-2\pi i)^s}{N^s(s-1)!}\sum_{n\ge 1}  n^{s-1}\eta^{a n} q^n}, & \hbox{if $s\ge 2$.}
  \end{array}
\right.
\end{equation}
\end{lem}
\begin{proof}
The well-known Lipschitz formula implies that for all $k\ge 2$
\begin{equation*}
\sum_{n\in \Z} \frac{1}{(x+ n)^k}= \frac{(-2\pi i)^k}{(k-1)!}
    \sum_{n\ge 1} n^{k-1}  e^{2 \pi i nx}.
\end{equation*}
Thus by setting $x=\tau+a/N$ we get
\begin{equation*}
\Psi^a_k(N\tau) =\sum_{n\in \Z}\frac{1}{(N\tau+nN+a)^k}
=\frac{(-2\pi i)^k}{N^k(k-1)!}\sum_{n\ge 1} n^{k-1} \eta^{a n} e^{2 \pi i nx},
\end{equation*}
as desired.

Now we deal with the special case when $s=1$. In the Summation Theorem \cite[p.~305]{MarsdenHo1984}
we take $f(z)=1/(z+\tau')$ where $\tau'=(\tau+a)/N$. Then we get
\begin{equation*}
   \lim_{M\to\infty} \sum_{c\equiv a\ppmod{N}, |c|<M} \frac{1}{\tau+ c}
   =   \frac{1}{N}\lim_{M\to\infty} \sum_{n=-M}^M \frac{1}{\tau'+n}
   = -{\rm Res}_{z=-\tau'} \left(\frac{\pi \cot \pi z}{z+\tau'}\right).
\end{equation*}
Since
\begin{equation*}
\cot \pi z=(-i)\frac{1+e^{2\pi i z}}{1-e^{2\pi i z}}=-i-2i\sum_{n\ge 1} e^{2\pi i n z}
\end{equation*}
we have
\begin{equation*}
\Psi^a_1(N\tau)=\frac{\pi}{N} \cot \Big[\pi \Big(\tau+\frac{a}{N}\Big)\Big]
=-\frac{\pi i}{N}-\frac{2\pi i}{N}\sum_{n\ge 1} \eta^{a n} q^n.
\end{equation*}
This completes the proof of the lemma.
\end{proof}

\begin{cor}\label{cor:productPsi}
For any $\bfa=(a_1,\dots,a_d)\in(\Z/N\Z)^d$ and $\bfs=(s_1,\dots,s_d)\in\N^d$ set
\begin{equation}\label{defn:g}
g_\bfs^\bfa(\tau)=g_\bfs^{\bfa;N}(\tau)=\sum_{m_1>\dots>m_d>0} \prod_{j=1}^d \Psi^{a_j}_{s_j}(m_jN\tau).
\end{equation}
Then we have
\begin{equation}\label{equ:g}
g_\bfs^\bfa(\tau) = \frac{(-2\pi i)^{|\bfs|}}{N^{|\bfs|} (\bfs-1)!}\sum_{n=1}^{\infty} \gs_{\bfs-\bfone}^\bfa(n) q^n,
\end{equation}
where $|\bfs|=s_1+\cdots+s_d$, $(\bfs-1)!=\prod_{j=1}^d (s_j-1)!$,  and $\bfs-\bfone=(s_1-1,\dots,s_d-1)$.
\end{cor}

It is now easy to decompose the level $N$ Eisenstein series (of one variable)
into the following form:
\begin{equation}\label{equ:1varEisenstein}
     G_r^a(\tau)=\zeta_N^a(r)+g_r^a(\tau)
\end{equation}
for any positive integer $r\ge 3$ and $a\in\Z/N\Z$. So we can
define two extensions of $G_r^a(\tau)$ as follows.
\begin{defn}\label{defn:extendG}
Let $\sharp=\sha$ or $\ast$. For all $s\ge 1$, we define
\begin{equation}\label{equ:extendG}
G^{a}_{s;\sharp}(\tau)=\zeta_{N;\sharp}^a(s)+g_s^a(\tau).
\end{equation}
\end{defn}
Notice that the definition is independent of whether $\sharp=\sha$ or $\ast$
by \eqref{equ:ZN1regval}.

The following theorem is the key to the double shuffle relations
satisfied by the double Eisenstein series at level $N$.

\begin{thm} \label{thm:fourierexpansion}
The Fourier series expansion of $G_{r,s}^{a,b}(\tau)$ for 
$r\geq 3, s\geq 2$ is given by
\begin{multline}\label{equ:expressGbyZetaAndg}
G_{r,s}^{a,b}(\tau)= \zeta^{a,b}_{N}(r, s)+  g^{a}_{r}(\tau)\zeta^b_{N}(s)
    + g_{r,s}^{a,b}(\tau)\\
+ \sum_{\substack{h+p=r+s\\ h\geq 1, p\geq \min\{r,s\}}}  \zeta^{a-b}_{N}(p)\left[(-1)^s\binom{p-1}{s-1} g^a_{h}(\tau)
    + (-1)^{p-r}\binom{p-1}{r-1} g^b_{h}(\tau)\right] .
\end{multline}
\end{thm}

\begin{proof}
Our proof follows the lines of that of \cite[Theorem 6]{GKZ2006}. We decompose $G_{r,s}^{a,b}(\tau)$ into the sum of the following four types: (1) $m=n=0,$ (2) $m>n=0,$ (3) $m=n>0,$ and (4) $m>n >0.$

(1) $m=n=0$. It gives rise to exactly
\begin{equation*}
\sum_{\substack{c>d>0\\ c\equiv a,\, d\equiv b\ppmod{N}}} \frac{1}{c^rd^s} = \zeta^{a,b}_{N}(r, s).
\end{equation*}

(2) $m>n=0$. We are looking at
\begin{equation*}
\sum_{\substack{m>0,\ d>0 \\ c\equiv a,\, d\equiv b\ppmod{N}}} \frac{1}{(mN\tau+c)^rd^s}
= \sum_{m>0}\Psi_r^a(mN\tau) \sum_{\substack{ d>0 \\ d\equiv b\ppmod{N}}} \frac{1}{d^s}
=g^a_{r}(\tau) \zeta^b_N(s).
\end{equation*}

(3) $m=n>0$. Then we write
\begin{equation*}
\sum_{\substack{m>0\\ c\equiv a,\, d\equiv b\ppmod{N}}}\frac{1}{(mN\tau +c)^r(mN\tau + d)^s}= \sum_{m>0} \Psi_{r,s}^{a,b}(mN\tau).
\end{equation*}
Next we compute $\Psi_{r,s}^{a,b}(\tau).$ Using the partial fraction
\begin{equation*}
\frac{1}{(\tau +c)^r(\tau + d)^s} =\sum_{\substack{h+p=r+s\\ h\geq 1, p\geq \min\{r,s\}}} \left[\frac{(-1)^s\binom{p-1}{s-1}}{(c-d)^p(\tau+c)^h}+\frac{(-1)^{p-r}\binom{p-1}{r-1}}
{(c-d)^p(\tau+d)^h}\right],
\end{equation*}
we obtain
\begin{multline*}
\Psi_{r,s}^{a,b}(\tau) = \sum_{\substack{c>d\\ c\equiv a,\, d\equiv b\ppmod{N}}} \frac{1}{(\tau+c)^r(\tau+d)^s}\\
= \sum_{\substack{h+p=r+s,\, c>d\\ c\equiv a,\, d\equiv b\ppmod{N}}}\left[(-1)^s\binom{p-1}{s-1}\frac{1}{(c-d)^p(\tau+c)^h} + (-1)^{p-r}\binom{p-1}{r-1}\frac{1}{(c-d)^p(\tau+d)^h}\right]\\
= \sum_{\substack{h+p=r+s\\ h\geq 1, p\geq \min\{r,s\}}}\zeta^{a-b}_N(p)\left[(-1)^s\binom{p-1}{s-1} \Psi^a_h(\tau)
    + (-1)^{p-r}\binom{p-1}{r-1} \Psi^b_h(\tau)\right].
\end{multline*}
Hence
\begin{multline*}
\sum_{\substack{m>0\\ c\equiv a,\, d\equiv b\ppmod{N}}}\frac{1}{(mN\tau +c)^r(mN\tau + d)^s}
    = \sum_{m>0} \Psi_{r,s}^{a,b}(mN\tau)\\
=\sum_{\substack{h+p=r+s\\ h\geq 1, p\geq \min\{r,s\}}}\zeta^{a-b}_{N}(p)\left[(-1)^s\binom{p-1}{s-1} g^a_{h}(\tau)
    + (-1)^{p-r}\binom{p-1}{r-1} g^b_{h}(\tau)\right] .
\end{multline*}
Here the special case $h=1$ has to be treated carefully by using \eqref{equ:Psi}.

(4) $m>n>0$.  We have
\begin{equation*}
 \sum_{\substack{m>n>0\\ c\equiv a,\, d\equiv b\ppmod{N}}}\frac{1}{(mN\tau +c)^r(nN\tau + d)^s}
 = g_{r,s}^{a,b}(\tau) .
\end{equation*}
The theorem follows by summing up the above four parts.
\end{proof}

Motivated by \eqref{equ:expressGbyZetaAndg} we now have
the extension of the double Eisenstein series of level $N$
to following regularized form.
\begin{defn}\label{defn:extendGrsab}
Let $\sharp=\sha$ or $\ast$. Then for all $r,s\ge 1$,
we define
\begin{multline}\label{equ:extenddblG}
G^{a,b}_{r,s;\sharp}(\tau)= \zeta^{a,b}_{N;\sharp}(r, s)+  g^{a}_{r}(\tau)\zeta^b_{N;\sharp}(s)
    + g_{r,s}^{a,b}(\tau)\\
+ \sum_{\substack{h+p=r+s\\ h\geq 1, p\geq \min\{r,s\}}}  \zeta^{a-b}_{N;\sharp}(p)\left[(-1)^s\binom{p-1}{s-1} g^a_{h}(\tau)
    + (-1)^{p-r}\binom{p-1}{r-1} g^b_{h}(\tau)\right] .
\end{multline}
\end{defn}

\begin{rem}
Unlike Definition~\ref{defn:extendG} this definition of
$G^{a,b}_{r,s;\sharp}(\tau)$ depends on the choice of
the regularization scheme $\sharp$.
Moreover, because of \eqref{equ:ZN1regval}
the dependence only appears in the constant term
$\zeta^{a,b}_{N;\sharp}(r,s)$.
\end{rem}

If $s,r\ge 3$ and $a,b\in\Z/N\Z$, then it is
not hard to show
\begin{equation}\label{equ:dblGstuffleForm}
G^a_{r}(\tau)G^b_{s}(\tau) =G_{r,s}^{a,b}(\tau)
            +G_{s,r}^{b,a}(\tau)+\gd_{a,b} G^a_{r+s}(\tau)
\end{equation}
which follows easily by the definition. But
\begin{equation*}
G^a_{r}(\tau)G^b_{s}(\tau)\ne \sum_{\substack{i+j=r+s\\  i,j\ge 1}} \left(\binom{i-1}{r-1} G^{a+b,b}_{i,j}(\tau)
    + \binom{i-1}{s-1} G^{a+b,a}_{i,j}(\tau) \right)
\end{equation*}
since the right-hand side has undefined terms.
Our goal is to give an extension of these double shuffle
relations to the case $r,g\ge 1$ and $(r,s)\ne (1,1)$
by using a complete version of the zeta values and
Eisenstein series of level $N$.

\section{Decomposition of the zeta values at level $N$}
In this section we break $\zeta_{N}^a(s)$ into two parts, one
of which is inspired by its complete version defined as follows.
For all positive integers $n$ and $a\in\Z/N\Z$, we set
\begin{equation}\label{equ:completeVersion}
\frakz_{N}^{a}(n)= \frac12\sum_{\substack{k\in\Z_{\ne 0} \\ k\equiv a\ppmod{N}}} \frac{1}{k^n}
 =\frac12\lim_{M\to\infty} \sum_{\substack{0<|k|<M \\ k\equiv a\ppmod{N}}} \frac{1}{k^n}
\end{equation}
by using the Cauchy principal value.
This infinite series converges absolutely for $n\ge 2$ and conditionally for $n=1$.
This complete version of $\zeta_{N}^a(s)$ clearly satisfies the stuffle relations
as those given by \eqref{equ:stuffle}.

\begin{rem}
When the level $N=2$, the decomposition of $\zeta_{N}^a(n)$
corresponds to the decomposition of it into the the
Bernoulli number part and non-Bernoulli number part.
See \cite[page 1103]{KanekoTa2013}. If $a\equiv 0 \pmod{2}$,
then the non-Bernoulli number part is essentially the Riemann
zeta values at odd integers.
\end{rem}

To extract more information from $\frakz_{N}^a(n)$ and find its relation
to $\zeta_{N}^a(n)$ we now recall that
the $n$-th Bernoulli periodic function $\bar{B}_n(x)$,
(see \cite[p.~267]{Apostol1976})
has the series expansion
\begin{equation}\label{equ:BernPolyExpansion}
\bar{B}_n(x)=-\frac{n!}{(2\pi i)^n}\sum_{k\in \bbZ_{\ne 0}} \frac{e^{2\pi i k x}}{k^n}
=-\frac{n!}{(2\pi i)^n}\lim_{M\to\infty} \sum_{0<|k|<M} \frac{e^{2\pi i k x}}{k^n},
\end{equation}
which converges absolutely for $n\ge 2$ and conditionally for $n=1$.
It is related to the Bernoulli polynomials by $\bar{B}_n(x)=B_n(\{x\})$
where $\{x\}$ is the fractional part of $x$, except for $n=x=1$ when
$B_1(\{1\})=B_1(0)=-1/2$ while
\begin{equation}\label{equ:barB11=0}
\bar{B}_1(1)=-\frac{1}{2\pi i}\lim_{M\to\infty} \sum_{0<|k|<M} \frac{1}{k}=0
\end{equation}
by symmetry.
The following identify motivates our definition of the constant
term of the generating series for the level $N$ Eisenstein series.

\begin{prop}\label{prop:HurwitzZeta}
For all $n\ge 1$ and $a\in\Z/N\Z$, we have
\begin{equation}\label{equ:averageZetaaAnd-a}
\frakz_{N}^{a}(n)=
-\frac{(2\pi i)^n}{2N\cdot n!} \sum_{l=1}^{N} \exp\left(-\frac{2\pi i l a}{N}\right)\bar{B}_n\left(\frac{l}{N}\right).
\end{equation}
\end{prop}
\begin{proof}
This follows quickly from the identity
\begin{equation*}
 \sum_{l=1}^{N} \eta_N^{l m}=
\left\{
  \begin{array}{ll}
    N, & \hbox{if $N|m$;} \\
    0, & \hbox{otherwise,}
  \end{array}
\right.
\end{equation*}
for any integer $m$.
\end{proof}

\begin{cor}\label{cor:vanishSine}
Let $1\leq a \leq N$. Then for all $r\ge 1$
\begin{equation*}
\sum_{l=1}^{N} \sin\left(\frac{2\pi l a}{N}\right)\bar{B}_{2r}\left(\frac{l}{N}\right)
=\sum_{l=1}^{N} \cos\left(\frac{2\pi l a}{N}\right)\bar{B}_{2r+1}\left(\frac{l}{N}\right)=0.
\end{equation*}
\end{cor}
\begin{proof}
Notice that under $a\to N-a$ the left-hand side of \eqref{equ:averageZetaaAnd-a}
is invariant if $n$ is even and changes the sign if $n$ is odd.
It also follows from the fact that $\bar{B}_{2r}(1-x)=\bar{B}_{2r}(x)$
and $\bar{B}_{2r+1}(1-x)=-\bar{B}_{2r+1}(x)$ for all $r\ge 0$.
\end{proof}

The following corollary provides the exact
relation between $\zeta_{N}^a$ and $\frakz_{N}^a$.
\begin{cor}\label{cor:tzN}
For all positive integers $n$ and $a\in\Z/N\Z$, we have
\begin{equation*}
\frakz_{N}^{a}(n)=
\frac{1}{2}\Big(\zeta_{N;\sharp}^{a}(n)+(-1)^n\zeta_{N;\sharp}^{-a}(n)\Big),
\end{equation*}
where $\sharp=\sha$ or $\ast$, and
when $n=1$ the right-hand side is defined by
\eqref{equ:ZN1regval}.
\end{cor}

\begin{proof}
Suppose $n\ge 2$ first. Then we can break the sum in
\eqref{equ:completeVersion} into two parts, one with positive
indices which produces $(2\pi i)^{-n}\zeta_{N}^{a}(n)$ and the other negative
which leads to $(-2\pi i)^{-n}\zeta_{N}^{-a}(n)$.
For $n=1$ the corollary follows easily from
Proposition~\ref{prop:HurwitzZeta} by using the fact that
\begin{equation*}
    Li_1(-e^{2\pi i \theta})-Li_1(e^{2\pi i \theta})= 2\pi i\theta-\pi i=2\pi i \bar{B}_1(\theta)
\quad \forall \theta\in(0,1) .
\end{equation*}
Notice that the $l=N$ term in the sum
on the right-hand side of \eqref{equ:averageZetaaAnd-a}
vanishes by \eqref{equ:barB11=0}.
We leave the details to the interested reader.
\end{proof}

Proposition~\ref{prop:HurwitzZeta}
leads us to the following definition if we follow the guideline that
the constant term of the multiple Eisenstein series
(even for the regularized values) should be closely related to
the multiple zeta values, at any level.
\begin{defn}\label{defn:gb}
For $n>0$ and $1\leq a \leq N$ we define
\begin{equation}\label{equ:betaDefn}
\gb^a_{n}=\gb^{a;N}_{n} = -\frac{1}{2N\cdot n!}\sum_{l=1}^N
    \exp\left(-\frac{2\pi i l a }{N}\right)B_n\left(\frac{l}{N}\right).
\end{equation}
\end{defn}
It is clear from Corollary \ref{cor:tzN}, Proposition~\ref{prop:HurwitzZeta},
and \eqref{equ:barB11=0} that
\begin{equation}\label{equ:gbtzRel}
\gb^a_{n}=(2\pi i)^{-n}\frakz_{N}^{a}(n)+\frac{\gd_{n,1}}{4N}.
\end{equation}

Let $\sharp=\sha$ or $\ast$. For all $r\geq 1, s\geq 1$
and $a,b\in\Z/N\Z$, we define
\begin{align*}
\tg^a_s(q)=&\, (2\pi i)^{-s}g^a_s(\tau),  \quad
\bargb^a_r= \frac{(2\pi i)^{-r}}{2}\Big(\zeta_{N;\sharp}^{a}(r)
    -(-1)^r\zeta_{N;\sharp}^{-a}(r)\Big)-\frac{\gd_{r,1}}{4N},  \\
I_{r,s}^{a,b}(q)=&\, \tg^{a}_{r}(q)\bargb^b_s
 + \sum_{\substack {h+p=r+s\\ h\geq 1}}  \bargb^{a-b}_p
 \left[(-1)^s\binom{p-1}{s-1} \tg^a_{h}(q)
    + (-1)^{p-r}\binom{p-1}{r-1} \tg^b_{h}(q)\right] .
\end{align*}\noindent
Notice that the quantities defined above
are independent of whether
$\sharp=\sha$ or $\ast$ according to \eqref{equ:ZN1regval}.
The following result generalizes \cite[Lemma 3]{KanekoTa2013}.

\begin{prop} \label{prop:fourierexpansion}
For any integer $k\ge 2$ and $a,b\in\Z/N\Z$, define the generating series
\begin{equation*}
    \frakI_k^{a,b}(X,Y)=\sum_{r+s=k} I_{r,s}^{a,b}(q) X^{r-1}Y^{s-1}.
\end{equation*}
If $k\ge 3$ then
\begin{multline}\label{equ:Irscase}
 \sum_{\substack{h+p=k\\ h\geq 1}}
  \Big(X^{h-1}Y^{p-1}\tg^{a}_{h}(q)\bargb^b_p
  + Y^{h-1}X^{p-1} \tg^{b}_{h}(q)\bargb^a_p\Big) \\
  = \frakI_k^{a,b}(X,Y)+\frakI_k^{b,a}(Y,X)
  =\frakI_k^{a+b,a}(X+Y,X)+\frakI_k^{a+b,b}(X+Y,Y).
\end{multline}
If $k=2$ (i.e., $r=s=1$), then we have
\begin{equation}\label{equ:I11case}
\tg^{a}_{1}(q)\bargb^b_1+\tg^{b}_1(q)\bargb^a_1= I_{1,1}^{a+b,b}+I_{1,1}^{a+b,a}.
\end{equation}
\end{prop}
\begin{proof}
We first see that
\begin{align}
&\,  \frakI_k^{a, b}(X, Y) = \sum_{r+s=k} I_{r,s}^{a, b}(q) X^{r-1}Y^{s-1}\notag\\
= &\, \sum_{\substack{h+p=k\notag\\ h\geq 1}} \tg^{a}_{h}(q)\bargb^b_p X^{h-1}Y^{p-1} +
 \sum_{h+p=k} \bargb^{a-b}_{p}(p)\cdot \notag\\
 &\, \left[\tg_h^a(q)\left(\sum_{r+s=k}(-1)^s\binom{p-1}{s-1}X^{r-1}Y^{s-1}\right) + \tg_h^b(q)\left(\sum_{r+s=k}(-1)^{p-r}\binom{p-1}{r-1}X^{r-1}Y^{s-1}\right)\right]\notag\\
= &\,  \sum_{\substack{h+p=k\notag\\ h\geq 1}} \tg^{a}_{h}(q)\bargb^b_p X^{h-1}Y^{p-1}\notag\\
+&\,\sum_{h+p=k} \bargb^{a-b}_p\Big( \tg_h^b(q)Y^{h-1}(X-Y)^{p-1}-\tg_h^a(q)X^{h-1}(X-Y)^{p-1}\Big)\label{equ:vanishUnderSum}
\end{align}{}\noindent
from binomial expansion.
Now by considering even and odd $p$ we see that
the sum term of \eqref{equ:vanishUnderSum} has
no contribution in $\frakI_k^{a, b}(X, Y)+\frakI_k^{b, a}(Y, X).$
The last equality of the \eqref{equ:Irscase} is straight-forward
so we omit its proof.
Finally, \eqref{equ:I11case} follows easily by direct computation.
\end{proof}

\section{Double shuffle relations of double Eisenstein series at level $N$}
In this section we are going to define three power series
$E_{r,s}^{a,b}(q), P_{r,s}^{a,b}(q)$
and $E_k(q)$ which, together with $I_{r,s}^{a,b}(q)$ for the first one,
are complementary to the double zeta values,
product of the zeta values and the zeta values at level $N$, respectively.
The latter values, essentially the constant terms of the corresponding
Eisenstein series, satisfy the double shuffle
relations in Proposition~\ref{prop:dbsf}.

Write $q=\exp(2\pi i \tau)$ and define
{\allowdisplaybreaks
\begin{align*}
\tPsi^\bfa_\bfs(q) =&\, \tPsi^{\bfa;N}_\bfs(q) = (2\pi i)^{-|\bfs|} \Psi^{\bfa;N}_\bfs(N\tau), \quad
\tg^\bfa_\bfs(q)= \tg^{\bfa;N}_\bfs(q) =  (2\pi i)^{-|\bfs|}g^\bfa_\bfs(\tau),\\
(\tg^a_k)'(q)=&\,(\tg^{a;N}_k)'(q) = -\sum_{n=1}^{\infty} n\tPsi^a_{k+1}(q^n), \quad \Big(f_k'(q)=\frac{q}{Nk}\frac{d}{dq}f_k(q)\Big)\\
\gb_{r,s}^{a,b}(q)=&\, \tg^a_r(q)\gb^{b}_{s}+ \sum_{i+j=r+s}\gb^{a-b}_i\left[(-1)^s\binom{i-1}{s-1}\tg^a_j(q) + (-1)^{i-r}\binom{i-1}{r-1}\tg^{b}_j(q)\right],\\
\vep_{r,s}^{a,b}(q)=&\,\gd_{r,2}(\tg^{b}_{s})'(q) - \gd_{r,1}(\tg^{b}_{s-1})' (q)
+ \gd_{s,1}\big((\tg^a_{r-1})'(q) + \tg^a_{r}(q)\big) +N \gd_{r,1}\gd_{s,1} \gam_N^{a,b}(q),
\end{align*}}\noindent
where $\gam_N^{a,b}=\gam_N^{a,b}(q)$ can be defined by the 
procedure to be outlined in Proposition~\ref{prop:gams}.
Further we set
\begin{equation}\label{equ:f2}
 f^{a}_2=f^{a}_2(q)=\frac1{N^2}\sum_{n,u=1}^\infty  \eta^{au} n q^{nu}=\frac1{N^2}\sum_{m=1}^\infty \gk_1^a(m) q^m, \ \text{where}\
  \gk_1^a(m)=\sum_{nu=m}  \eta^{au} n.
\end{equation}

\begin{rem} \label{rem:gammas}
(i). To save space, in the rest of the paper we will always suppress the dependence
on $q$ in the $q$-series $\gam_N^{a,b}(q)$, $f^0_2(q)$, etc. Of course, they all depend on $N$.

(ii). We will see that the definition of $\gam_N^{a,b}$ is not unique. For example,
for small levels we may define  $\gam_N^{a,b}$ explicitly as follows.
For $N=1$: $\gam_1^{0,0}=f^0_2=\tg^0_2$. For $N=2$ we can set
\begin{equation}\label{equ:gamN=2}
  \gam_2^{0,0}=f^0_2= \tg^0_2,  \quad \gam_2^{0,1}=\gam_2^{1,0}=0,
\quad \gam_2^{1,1}=\tg^1_2-f^1_2.
\end{equation}
When $N=2$ our choice of the above are different from that of \cite{KanekoTa2013}
(see Remark~\ref{rem:KT2012}). For $N=3$ we may define
\begin{equation}\label{equ:gamN=3}
\begin{split}
    &   \gam_3^{a,a}=\tg^a_2-f^a_2 \quad  \text{and}\quad
    \gam_3^{a,0}=f^0_2+f^a_2-\gam_3^{a,a} \quad \text{ for }a=1,2,  \\
    &  \gam_3^{0,0}=\tg^0_2,\quad \gam_3^{0,1}=-\gam_3^{1,0}, \quad \gam_3^{0,2}=-\gam_3^{2,0}, \quad \gam_3^{1,2}=\gam_3^{2,1}=0.
\end{split}
\end{equation}
\end{rem}

\begin{defn}\label{defn:EandP}
We define
{\allowdisplaybreaks
\begin{align*}
E^{a}_{k}(q)&=E^{a;N}_{k}(q)=\left\{
  \begin{array}{ll}
   \tg^{a}_{k}(q), & \hbox{if $k>2$;} \\
    0, & \hbox{if $k\le 2$;}
  \end{array}
\right.\\
E_{r,s}^{a,b}(q) &= E^{a, b;N}_{r,s}(q) = \tg_{r,s}^{a,b}(q)+ \gb_{r,s}^{a,b}(q)+\frac{\vep_{r,s}^{a,b}(q)}{2N},  \quad r,s\ge 1;\\
P_{r,s}^{a,b}(q)&= P^{a, b;N}_{r,s}(q)= \tg^a_{r}(q)\tg^{b}_{s}(q) + \gb^a_{r}\tg^{b}_{s}(q) + \gb^{b}_{s}\tg^{a}_{r}(q) \\
 &\hskip2cm + \frac{\gd_{r,2}(\tg^b_{s})'(q) +\gd_{s,2}(\tg^{a}_{r})'(q)}{2N}+\gd_{r,1}\gd_{s,1} \gl_N^{a,b}, \quad r,s\ge 1.
\end{align*}}\noindent
The quantities  $\gl_N^{a,b}=\gl_N^{a,b}(q)=\gl_N^{b,a}(q)$ will be 
defined by the procedure to be outlined in Proposition~\ref{prop:gams}
together with  $\gam_N^{a,b}$'s.
\end{defn}
For example, to be compatible with
\eqref{equ:gamN=2} we set
\begin{equation}\label{equ:glN=2}
 \gl_2^{0,0}= \gl_2^{1,0}=0,\quad \gl_2^{1,1}=-f_2^1,
\end{equation}
and  to be compatible with \eqref{equ:gamN=3} we set
\begin{equation}\label{equ:glN=3}
 \gl_3^{0,0}= \gl_3^{1,0}= \gl_3^{2,0}=\gl_3^{2,1}=0,\quad \gl_3^{1,1}=-f_2^1, \quad \gl_3^{2,2}=-f_2^2.
\end{equation}

Roughly speaking, the double Eisenstein series $G_{r,s}^{a,b}(\tau)$ is 
given by the sum of $\zeta_N^{a,b}(r,s)$ and $(2\pi i)^{r+s}E_{r,s}^{a,b}(q)$ 
(similar for depth one Eisenstein series $G^{a}_{r}(\tau)$)
while the product $G^{a}_{r}(\tau)G^{b}_{s}(\tau)$
is given by the sum of $\zeta_N^{a}(r)\zeta_N^{b}(s)$ and
$(2\pi i)^{r+s}P_{r,s}^{a,b}(q)$.
By Proposition~\ref{prop:dbsf}, the zeta function part already
satisfies the double shuffle relations. So to prove that similar relations
hold for Eisenstein series it suffices to prove the next result which
generalizes \cite[Lemma 4]{KanekoTa2013}.

\begin{thm}\label{thm:EisenSeriesDS}
Let $N$ be a positive integer. Then there are suitable choices of
$\gam_N^{a,b}$ and $\gl_N^{a,b}$, $0\le a,b<N$, such that they provide
the solution to the linear system
\begin{equation}\label{equ:gamStuffleAllN}
\gam_N^{a,a}-\gl_N^{a,a}=\tg^{a}_2,\quad
\gam_N^{a,b}+\gam_N^{b,a}-2\gl_N^{a,b}=0\quad \forall a\ne b \in\Z/N\Z,
\end{equation}
together with
\begin{equation}\label{equ:gamShuffleAllN}
\gam_N^{a+b,a}+\gam_N^{a+b,b}-2\gl_N^{a,b} = f^{a}_{2}+f^{b}_{2}\quad
\forall a,b\in\Z/N\Z.
\end{equation}
Consequently, for all $r,s\ge 1$
the three power series $E_{r,s}^{a,b}(q)$, $P_{r,s}^{a,b}(q)$
and $E^{a}_k(q)$ satisfy the double shuffle relation at level $N$:
\begin{align}\label{equ:stuffleEisenstein}
P_{r,s}^{a,b}(q)=&\,E_{r,s}^{a,b}(q)+E_{s,r}^{b,a}(q)+\gd_{a,b} E^{a}_{r+s}(q)\\
=&\,\sum_{\substack{i+j=r+s\\ i, j\ge 1}} \left(\binom{i-1}{r-1} E_{i,j}^{a+b,b}(q)
    + \binom{i-1}{s-1} E_{i,j}^{a+b,a}(q) \right), \label{equ:shuffleEisenstein}
\end{align}
\end{thm}

\begin{proof}
It suffices to show that there are suitable choices of $\gam_N^{a,b}$ and
$\gl_N^{a,b}$ satisfying \eqref{equ:gamStuffleAllN} and
\eqref{equ:gamShuffleAllN} such that the generating functions
\begin{align*}
 \frakE^{a} (X)=&\sum_{k\geq 1} E^a_k(q)X^{k-1},\\
 \frakE^{a,b} (X,Y)=& \sum_{r, s\geq 1} E_{r,s}^{a,b}(q) X^{r-1}Y^{s-1},\\\
 \frakP^{a,b} (X,Y)=&\sum_{r, s\geq 1} P_{r,s}^{a,b}(q)X^{r-1}Y^{s-1},
\end{align*}\noindent
satisfy the double shuffle relation:
\begin{align} \label{equ:EseriesStuffle}
  \frakP^{a,b} (X,Y)=&\frakE^{a,b} (X,Y)+\frakE^{b,a} (Y,X)+\gd_{a,b}\frac{\frakE^{a} (X)-\frakE^{a} (Y)}{X-Y},\\
      =&\frakE^{a+b,b} (X+Y,Y)+\frakE^{a+b,a} (X+Y,X).\label{equ:EseriesShuffle}
\end{align}

We first calculate the generating functions of the above defined power series. Set
{\allowdisplaybreaks
\begin{align*}
\tg^a(X)=&\sum_{k=1}^\infty \tg_k^a(q)X^{k-1} = -\frac{1}{N}\sum_{n=1}^{\infty}
    \eta^{an}e^{-\frac{nX}{N}}\frac{q^{n}}{1-q^n},\\
(\tg^a)'(X)=&\sum_{k=1}^\infty (\tg_k^a)'(q)X^{k-1} =\frac{1}{NX}\left(\sum_{n=1}^{\infty}
   \eta^{an}e^{-\frac{Xn}{N}}\frac{q^n}{(1-q^n)^2}-N^2 f^{a}_2(q)\right),\\
\tg^{a,b}(X,Y)=&\sum_{r,s=1}^\infty \tg_{r,s}^{a,b}(q)X^{r-1}Y^{s-1} = \frac{1}{N^2}\sum_{m,n\geq 1}
    \eta^{am+bn}e^{\frac{-mX-nY}{N}}\frac{q^m}{1-q^m}\frac{q^{m+n}}{1-q^{m+n}},\\
\gb^{a,b}(X,Y)=&\sum_{r,s=1}^\infty \gb_{r,s}^{a,b}(q)X^{r-1}Y^{s-1}
    = (\tg^b(Y)-\tg^a(X))\gb^{a-b}(X-Y)+ \tg^{a}(X)\gb^b(Y),\\
\vep^{a,b}(X,Y) =&\sum_{r,s=1}^\infty \vep_{r,s}^{a,b}(q)X^{r-1}Y^{s-1}
    =X(\tg^b)'(Y) - Y(\tg^b)'(Y) + X(\tg^a)'(X) + \tg^a(X)\\
    &\phantom{\sum \vep_{r,s}^{a,b}(q)X^{r-1}Y^{s-1}} -(\tg^b_0)'(q) +(\tg^a_0)'(q) + N\gam_N^{a,b}(q).
\end{align*}}

\begin{rem}
(i). Notice that $f^0_2(q)=\tg^0_2(q)$ but in general $f^{a}_2(q)\ne \tg^{a}_2(q).$

(ii). Notice also that
\begin{equation*}
(\tg^a_0)'(q)=-\sum_{n=1}^{\infty} n\tPsi^a_1(q^n)=
\frac1N\sum_{n=1}^{\infty} n\sum_{u=1}^{\infty} \eta^{au}  q^{un}
=Nf^a_2(q).
\end{equation*}
\end{rem}

Turning back to the proof of Theorem \ref{thm:EisenSeriesDS}, 
by the definition, we have
{\allowdisplaybreaks
\begin{align*}
\gb^a(X) &=  \sum_{k=1}^\infty \gb^a_k X^{k-1}
= -\frac{1}{2NX}\sum_{l=1}^N \sum_{n=1}^\infty \frac{X^n}{n!}
    e^{-2\pi la i/N} B_n\left(\frac{l}{N}\right)\\
&=-\frac{1}{2NX}\sum_{l=1}^N  e^{-2\pi la i/N}\left(\frac{Xe^{Xl/N}}{e^X-1}-1\right) \\
&= -\frac{1}{2N}\sum_{l=1}^N  \frac{e^{(X-2\pi  a i)l/N} }{e^X-1} + \frac{\gd_{a,0}}{2X}
=-\frac{1}{2N} \frac{1}{e^{(X-2\pi  a i)/N}-1} + \frac{\gd_{a,0}}{2X}.
\end{align*}}\noindent
Then
\begin{align*}
\frakE^a(X)&= \tg^a(X)-X\tg^{a}_{2}(q)-\tg^{a}_{1}(q),\\
\frakE^{a,b}(X,Y)&= \tg^{a,b}(X,Y)+\gb^{a,b}(X,Y) + \frac{1}{2N}\vep^{a,b}(X,Y),\\
\frakP^{a,b}(X,Y)&= \tg^a(X)g^b(Y)+ \gb^a(X)\tg^b(Y)+ \gb^b(Y)\tg^a(X) \\
    & +\frac{1}{2N}(X(\tg^b)'(Y) + Y(\tg^a)'(X))+M_N\gd_{a,b}(\gd_{a,0}-1)f_2^a(q).
\end{align*}\noindent
Now we compute $\frakE^{a,b}(X,Y)+ \frakE^{b, a}(Y,X).$
Straight-forward computation yields that
{\allowdisplaybreaks
\begin{align}
& \begin{split}
\tg^{a,b}(X,Y)+ \tg^{b, a}(Y,X) &= \tg^a(X)\tg^b(Y)- \frac{\tg^a(X)+\tg^b(Y)}{2N}  \\
&-\frac{1}{2N}\coth \left(\frac{X-Y-2\pi i(a-b)}{2N}\right)(\tg^a(X)-\tg^b(Y)),
\end{split} \label{equ:gXY} \\
& \begin{split}
\vep^{a,b}(X,Y)+\vep^{b, a}(Y,X)&= X(\tg^b)'(Y) + Y(\tg^a)'(X)+\tg^a(X)+\tg^b(Y) \ \hskip1cm \ \\
 &+N\gam_N^{a,b}(q)+N\gam_N^{b,a}(q).
\end{split} \label{equ:vepXY}
\end{align}}\noindent
On the other hand, if we set $\theta=(X-Y-2\pi (a-b) i)/N$ then
{\allowdisplaybreaks
\begin{align}
\gb^{a,b}(X,Y)+\gb^{b, a}(Y,X)&=
 (\tg^b(Y)-\tg^a(X))\gb^{a-b}(X-Y)+ \tg^{a}(X)\gb^b(Y) \notag \\
& - (\tg^b(Y)-\tg^a(X))\gb^{b-a}(Y-X)+ \tg^{b}(Y)\gb^a(X)  \notag \\
=&  \frac{\tg^b(Y)-\tg^a(X)}{2N}
 \left(\frac{1}{e^{-\theta}-1}-\frac{1}{e^\theta-1} \right)\notag \\
& + \tg^{a}(X)\gb^b(Y)+ \tg^{b}(Y)\gb^a(X)-\gd_{a,b}\frac{\tg^a(X)-\tg^a(Y)}{X-Y}\notag  \\
=& \frac{\tg^a(X)-\tg^b(Y)}{2N} \coth\left(\frac{X-Y-2\pi i(a-b)}{2N} \right)  \notag \\
 &+ \tg^a(X)\gb^b(Y) + \tg^b(Y)\gb^a(X)-\gd_{a,b}\frac{\tg^a(X)-\tg^a(Y)}{X-Y}.\label{equ:gbXY}
\end{align}}\noindent
Adding up \eqref{equ:gXY}, $\frac1{2N}\times$\eqref{equ:vepXY} and \eqref{equ:gbXY}
we can derive \eqref{equ:EseriesStuffle} quickly if the conditions
in \eqref{equ:gamStuffleAllN} are satisfied.
Similarly,
\begin{align}
& \tg^{a+b,a}(X+Y,X)+\tg^{a+b,b}(X+Y,Y) \notag\\
=& \frac{1}{N^2}\sum_{m\ne n\geq 1}
    \eta^{am+bn}e^{\frac{-mX-nY}{N}}\frac{q^m}{1-q^m}\frac{q^{n}}{1-q^n},\notag\\
=&\tg^a(X)\tg^b(Y)-\frac{1}{N^2}\sum_{n\geq 1}
    \eta^{(a+b)n}e^{\frac{-(X+Y)n}{N}}\left(\frac{q^{n}}{(1-q^n)^2}-\frac{q^{n}}{1-q^n}\right), \notag\\
=&\tg^a(X)\tg^b(Y)-\frac{X+Y}{N}(\tg^{a+b})'(X+Y)- f^{a+b}_{2}(q) -\frac1N \tg^{a+b}(X+Y), \label{equ:gX+YY}
\end{align}\noindent
and
\begin{align}
  & \vep^{a+b,a}(X+Y,X)+\vep^{a+b,b}(X+Y,Y) \notag\\
=& X(\tg^b)'(Y) +Y(\tg^a)'(X) + 2(X+Y)(\tg^{a+b})'(X+Y) + 2\tg^{a+b}(X+Y) \notag\\
+&2Nf^{a+b}_2(q)-Nf^{a}_2(q)-Nf^{b}_2(q)+ N\gam_N^{a+b,a}(q)+ N\gam_N^{a+b,b}(q). \label{equ:vepX+YY}
\end{align}\noindent
Further
\begin{align}
  & \gb^{a+b,a}(X+Y,X)+\gb^{a+b,b}(X+Y,Y) \notag\\
 =&(\tg^b(Y)-\tg^{a+b}(X+Y))\gb^a(X)+ \tg^{a+b}(X+Y)\gb^b(Y) \notag\\
 +&(\tg^a(X)-\tg^{a+b}(X+Y))\gb^{b}(Y)+ \tg^{a+b}(X+Y)\gb^a(X) \notag\\
 =&\tg^b(Y)\gb^a(X)+\tg^a(X)\gb^{b}(Y). \label{equ:gbX+YY}
\end{align}\noindent
Adding up \eqref{equ:gX+YY}, $\frac1{2N}\times$\eqref{equ:vepX+YY} and \eqref{equ:gbX+YY}
we can prove \eqref{equ:EseriesShuffle} if the conditions in \eqref{equ:gamShuffleAllN}
are satisfied.

To complete the proof of the theorem we now need to show that the system
\eqref{equ:gamStuffleAllN} together with \eqref{equ:gamShuffleAllN} has at least
one set of solutions of  $\gam_N^{a,b}$ and $\gl_N^{a,b}$ ($0\le a,b<N$)  in terms of
$f_2^a(q)$ and $\tg_2^a(q)$ ($0\le a<N$).
Essentially as a linear algebra problem this will be solved
in Proposition~\ref{prop:gams} in the last section of this paper.
This completes the proof of the theorem.
\end{proof}

Let  $\tz_{N;\sharp}^a(r)=(2\pi i)^{-r}\zeta_{N;\sharp}^a(r)$,
$\tz_{N;\sharp}^{a,b}(r,s)=(2\pi i)^{-r-s}\zeta_{N;\sharp}^{a,b}(r,s)$,
and define $\tG_{r;\sharp}^{a}(q)$ and $\tG_{r,s;\sharp}^{a,b}(q)$ similarly.

\begin{thm} \label{thm:EisenSeriesDSforG}
Let $N$ be a positive integer.  Let $\sharp=\sha$ or $\ast$.
Then for all $a,b\in\Z/N\Z$ and $r,s\ge 1$ with $(r,s)\ne 1$, we have
\begin{align}
\tG_{r;\sharp}^a(q)\tG_{s;\sharp}^b(q)
=&\,\tG^{a,b}_{r,s;\ast}(q)+\tG^{b,a}_{s,r;\ast}(q)+\gd_{a,b} \tG^{a}_{r+s;\sharp}(q)+\frac{\gd_{s,1}\tg^{a}_{r}(q)+\gd_{r,1}\tg^{b}_{s}(q)}{2N}
\label{equ:stuffleEisensteinG} \\
=&\,\sum_{\substack{i+j=r+s\\ i, j\ge 1}} \left(\binom{i-1}{r-1} \tG_{i,j;\sha}^{a+b,b}(q)
    + \binom{i-1}{s-1} \tG_{i,j;\sha}^{a+b,a}(q) \right)
    +f_{r,s}^{a,b}(q) , \label{equ:shuffleEisensteinG}
\end{align}
where $f_{r,s}^{a,b}(q) =\binom{k-2}{s-1}\big((\tg^{a+b}_{k-2})'(q) + \tg^{a+b}_{k-1}(q)\big)/N$ with $k=r+s$. Moreover,
\begin{multline}\label{equ:G11case}
\tG_{1;\sharp}^{a}(q)\tG_{1;\sharp}^{b}(q)=\tG_{1,1;\sha}^{a+b,b}(q)
    +\tG_{1,1;\sha}^{a+b,a}(q)+\frac12(f_2^a+f_2^b)\\
    +\frac1{2N}\Big(2(\tg^{a+b}_0)'(q)+2\tg^{a+b}_1(q)
        -(\tg^{a}_0)'(q)-(\tg^{b}_0)'(q)\Big).
\end{multline}
\end{thm}

\begin{proof}
By Corollary \ref{cor:tzN} and \eqref{equ:gbtzRel}, we have
\begin{equation*}
   \bargb^a_{s}+\gb^a_{s} =\tz^a_{N}(s) .
\end{equation*}
So by the definitions,
\begin{align}
\tG_{r;\sharp}^{a}(q)=&\, \tz_{N;\sharp}^a(r)+\tg_r^a(q), \label{equ:decomptG}\\
\tG_{r,s;\sharp}^{a,b}(q)=&\, \tz^{a,b}_{N;\sharp}(r, s)
+\tg_{r,s}^{a,b}(q) +I_{r,s}^{a,b}(q)+\gb_{r,s}^{a,b}(q).\label{equ:decomptGdbl}
\end{align}\noindent
Thus
\begin{align*}
&\, \tG_{r;\sharp}^{a}(q)\tG_{s;\sharp}^{b}(q)+
\frac{\gd_{r,2}(\tg^b_{s})'(q) +\gd_{s,2}(\tg^{a}_{r})'(q)}{2N}\\
=&\,  \tz_{N;\sharp}^a(r)\tz_{N;\sharp}^b(s)+P_{r,s}^{a,b}(q)+\bargb_r^a \tg_s^b(q)
+\bargb_s^b \tg_r^a(q)  \\
=&\,  \tz_{N;\ast}^{a,b}(r,s)+E_{r,s}^{a,b}(q)
+\tz_{N;\ast}^{b,a}(s,r)+E_{s,r}^{b,a}(q)+\gd_{a,b} \Big(\tz_{N}^a(r+s)+E^{a}_{r+s}(q)\Big)+I_{r,s}^{a,b}(q)+I_{s,r}^{b,a}(q) \\
=&\,\sum_{\substack{i+j=r+s\\ i,j\ge 1}}  \bigg[\binom{i-1}{r-1}
\Big(\tz^{a+b,b}_{N;\sha}(i,j)+E_{i,j}^{a+b,b}(q) +I_{i,j}^{a+b,b}(q)\Big) \\
&\, \qquad\quad + \binom{i-1}{s-1} \Big(\tz^{a+b,a}_{N;\sha}(i,j)+E_{i,j}^{a+b,a}(q) +I_{i,j}^{a+b,a}(q)\Big) \bigg]
\end{align*}\noindent
by Proposition~\ref{prop:dbsfreg}, Proposition~\ref{prop:fourierexpansion} and Theorem~\ref{thm:EisenSeriesDS}.
Note that $(r,s)\ne(1,1)$ just because $r+s\ge 3$ when using Proposition \ref{prop:fourierexpansion}. Now by the definition
\begin{equation*}
    \vep_{r,s}^{a,b}(q)+\vep_{s,r}^{b,a}(q)
    =\gd_{s,2}(\tg^{a}_{r})'(q)+\gd_{r,2}(\tg^{b}_{s})'(q)
+ \gd_{s,1}\tg^a_{r}(q)+ \gd_{r,1}\tg^b_{s}(q).
\end{equation*}
Hence \eqref{equ:stuffleEisensteinG}
follows from Definition~\ref{defn:EandP},
\eqref{equ:decomptG}, and \eqref{equ:decomptGdbl}.
For \eqref{equ:shuffleEisensteinG} we need to compute
\begin{align*}
&\,\sum_{\substack{i+j=r+s\\ i,j\ge 1}}  \bigg[\binom{i-1}{r-1}
 \vep_{i,j}^{a+b,b}(q)
 + \binom{i-1}{s-1}\vep_{i,j}^{a+b,a}(q)  \bigg]\\
=&\,\sum_{\substack{i+j=r+s\\ i,j\ge 1}}  \bigg[\binom{i-1}{r-1}
\Big(\gd_{i,2}(\tg^{b}_{j})'(q) - \gd_{i,1}(\tg^{b}_{j-1})' (q)
+ \gd_{j,1}\big((\tg^{a+b}_{i-1})'(q) + \tg^{a+b}_{i}(q)\big) \Big) \\
&\, \qquad\quad +\binom{i-1}{s-1}
\Big(\gd_{i,2}(\tg^{a}_{j})'(q) - \gd_{i,1}(\tg^{a}_{j-1})' (q)
+ \gd_{j,1}\big((\tg^{a+b}_{i-1})'(q) + \tg^{a+b}_{i}(q)\big) \Big) \bigg]\\
=&\, \gd_{r,2}(\tg^{b}_s)'(q)+\gd_{s,2}(\tg^{a}_r)'(q)+
\left[\binom{k-2}{r-1} +\binom{k-2}{s-1}\right]
\Big((\tg^{a+b}_{k-2})'(q) + \tg^{a+b}_{k-1}(q)\Big)
\end{align*}\noindent
where $k=r+s$. This yields \eqref{equ:decomptGdbl} immediately.

Finally, \eqref{equ:G11case} follows from direct computation using
\eqref{equ:I11case} and \eqref{equ:gamShuffleAllN}. This finishes
the proof of the theorem.
\end{proof}

\begin{rem}
When $N=1$ Theorem~\ref{thm:EisenSeriesDSforG} reduces to \cite[Theorem 7]{GKZ2006}.
When $N=2$ Theorem~\ref{thm:EisenSeriesDSforG} reduces to
\cite[Theorem 3]{KanekoTa2013} with some correction there.
\end{rem}

\section{A key relation on multiple divisor functions at level $N$}\label{sec: multipleDivisorFunction}
In this section we prove a key result on multiple divisor functions at level $N$,
which will be used in the next section.

Let $\gf$ be Euler's totient function. We first need a lemma concerning some special
power sums of roots of unity.

\begin{lem} \label{lem:KeyetaSum}
Let $N=\prod_{t=1}^r p_t^{k_t}$ and $\eta$ be a primitive $N$-th root of unity.
For $\ga_t\le k_t$, $t=1,\dots,r$ (but $\bga=(\ga_1,\dots,\ga_r)\ne (k_1,\dots,k_r)$) we write
\begin{equation*}
 J(\bga)=J_N(\ga_1,\dots,\ga_r)=\{1\le \iij<N: \ p_t^{\ga_t}|\!| \iij \quad \forall   t=1,\dots,r\}.
\end{equation*}
Then for any choice of $r$-tuple of non-negative integers $(\ell_1,\dots,\ell_r)$ we have
\begin{equation*} 
\sum_{\iij\in J(\ga_1,\dots,\ga_r)} \eta^{\iij \prod_{t=1}^r p_t^{\ell_t} }
 =\left\{
    \begin{array}{ll}
      0, & \hbox{if $\ell_t\le k_t-\ga_t-2$ for some $t\le r$;} \\
      \prod_{t\in I} (-p_t^{\ell_t})   \prod_{s\not\in I} \gf(p_s^{k_s-\ga_s}), & \hbox{if $C_I$ holds,}
    \end{array}
  \right.
\end{equation*}
where $C_I$ is the condition that there is $I\subseteq\{1,\dots,r\}$ such that
$\ell_t=k_t-\ga_t-1$ $\forall t\in I$ and $\ell_s\ge k_s-\ga_s$  $\forall s\not\in I$.
\end{lem}

\begin{proof}
Suppose $\xi$ is a primitive $p^k$-th root of unity for some prime $p$. Then
for all $\ga< k$ we have
\begin{equation}\label{equ:singlePrimePowerSum}
\sum_{p^\ga |\!| \iij, 1\le \iij< p^k } \xi^{\iij}=
\sum_{p^\ga | \iij, 1\le \iij< p^k} \xi^{\iij}
-\sum_{p^{\ga+1} | \iij, 1\le \iij< p^k} \xi^{\iij}
=\left\{
    \begin{array}{ll}
      -1, & \hbox{if $\ga=k-1$;} \\
      0, & \hbox{if $\ga<k-1$,}
    \end{array}
  \right.
\end{equation}
since for any divisor $D$ of $p^k$ we have
$$\sum_{D| \iij, 1\le \iij<p^k} \xi^{ \iij}=
\left\{
  \begin{array}{ll}
    -1, & \hbox{if $D<p^k$;} \\
    0, & \hbox{if $D=p^k$.}
  \end{array}
\right. $$
Let $N_t=p_t^{k_t}$ for all $t=1,\dots,r$. It is well-known that $\eta$ can be decomposed as $\eta=\prod_{t=1}^r \xi'_t$
where $\xi'_t$ is a primitive $N_t$-th root of unity for each $t$. Then $\xi_t=(\xi'_t)^{\prod_{s\ne t,1\le s\le r} p_s^{\ell_s}}$
is still a primitive $N_t$-th root of unity. By the Chinese Remainder Theorem it is easy to see that
\begin{equation*} 
\sum_{\iij\in J(\ga_1,\dots,\ga_r)} \eta^{\iij \prod_{t=1}^r p_t^{\ell_t} }
 = \prod_{t=1}^r \left(\sum_{\iij_t\in J_{N_t}(\ga_t)} \xi_t^{\iij_t  p_t^{\ell_t} } \right).
\end{equation*}
The lemma now follows from \eqref{equ:singlePrimePowerSum} and the fact that $|J_{N_t}(\ga_t)|=\gf(p_t^{k_t-\ga_t})$.
\end{proof}

Recall that for a $a\in\Z/N\Z$ we have defined the level $N$ divisor functions
$\gs_1^a(m)=\sum_{nu=m} \eta^{au} u$ and $\gk_1^a(m)=\sum_{nu=m} \eta^{au} n.$
\begin{thm} \label{thm:AllLevels}
Let $N=\prod_{\iij=1}^r p_\iij^{k_\iij} $ where $p_1,\dots,p_r$ are pairwise distinct prime
factors of $N$. Then for all $m\in\N$ we have
\begin{multline}\label{equ:AllLevels}
 \sum_{\gcd(N,\iij)=1, 1\le \iij<N}  \gs^\iij_1(m)- \gf(N)\gs^0_1(m) =\\
 \sum_{I \subseteq \{1,\dots,r\}} \left(  \prod_{t\in I}\gf(p_t^{k_t} ) \prod_{s\not\in I} p_s^{k_s}
\sum_{p_t|\iij \ \forall t\in I, p_s\nmid \iij \ \forall s\not\in I, 1\le \iij<N} \gk^\iij_1(m) \right) .
\end{multline}
\end{thm}
\begin{proof}
To save space we put $[r]=\{1,\dots,r\}$.
Let $e_t\ge 0$ for all $t\in [r]$ and assume $m= \prod_{t=1}^r p_t^{e_t} \prod_{\iij=r+1}^R p_\iij^{e_\iij}$ where $p_1,\dots,p_R$ are
pairwise distinct primes. Set $Q=\prod_{\iij=r+1}^R  (1+p_\iij+\cdots+p_\iij^{e_\iij})$ ($Q=1$ if none of $p_{r+1},\dots,p_R$ appears).

If $\ell_t\le e_t$ for all $t\in [r]$ then
\begin{align*}
 & \sum_{un=m,p_t^{\ell_t} |\!|u \  \forall t\in [r]} \ \sum_{\iij\in J(\ga_1,\dots,\ga_r)} \eta^{\iij u} u  =
 Q  \prod_{t=1}^r p_t^{\ell_t}  \sum_{\iij\in J(\ga_1,\dots,\ga_r)} \eta^{\iij  \prod_{t=1}^r p_t^{\ell_t} } .
\end{align*}\noindent
If $\ga_t> k_t$ for some $t\in[r]$ then $\sum_{\iij\in J(\ga_1,\dots,\ga_r)}  \gs^\iij_1(m)=0$  by the definition of $J$.
For any partition $[r]=\amalg \boldsymbol{\gL}=\gL_1\amalg \gL_2\amalg \gL_2$ with
$\boldsymbol{\gL}=(\gL_1, \gL_2, \gL_3)\ne (\emptyset,\emptyset,[r])$ we write
$\bga=(\ga_1,\dots,\ga_r) \vdash\boldsymbol{\gL}$
if $\ga_t=0$ for all $t\in\gL_1$, $1\le \ga_t<k_t$ for all $t\in\gL_2$ and $\ga_t=k_t$ for all $t\in\gL_3$.
We remove the case $\boldsymbol{\gL}=(\emptyset,\emptyset,[r])$ since
$(\ga_1,\dots,\ga_r)\ne (k_1,\dots,k_r)$.
For such $\bga$ we have by Lemma~\ref{lem:KeyetaSum}
{\allowdisplaybreaks
\begin{align*}
\sum_{\iij\in J(\ga_1,\dots,\ga_r)}  \gk^\iij_1(m)
=& Q\prod_{t\not\in\gL_3}\Big(-p_t^{e_t}+\gf(p_t^{k_t-\ga_t})\sum_{\ell_t=k_t-\ga_t}^{e_t}p_t^{e_t-\ell_t}\Big)
    \prod_{t\in\gL_3}\Big( \sum_{\ell_t=0}^{e_t} p_t^{e_t-\ell_t}\Big)\\
 = & Q\prod_{t\not\in\gL_3} \Big(-p_t^{k_t-\ga_t-1}\Big)  \prod_{t\in\gL_3}\frac{p_t^{e_t+1}-1}{p_t-1}   , \\
\sum_{\iij\in J(\ga_1,\dots,\ga_r)}  \gs^\iij_1(m)
=& Q\prod_{t\not\in\gL_3}\Big(-p_t^{2(k_t-\ga_t-1)}+ \gf(p_t^{k_t-\ga_t})\sum_{\ell_t=k_t-\ga_t}^{e_t}p_t^{\ell_t}\Big)
    \prod_{t\in\gL_3}\Big( \sum_{\ell_t=0}^{e_t} p_t^{\ell_t}\Big)\\
 = & Q\prod_{t\not\in\gL_3} \Big(-p_t^{k_t-\ga_t-1}(p_t^{e_t+1}-p_t^{k_t-\ga_t}-p_t^{k_t-\ga_t-1})\Big)  \prod_{t\in\gL_3}\frac{p_t^{e_t+1}-1}{p_t-1}   , \\
\end{align*}}\noindent
if $e_t\ge k_t-1$ for all $t\in[r]$. If $e_t<k_t-1$ then
$$\sum_{\iij\in J(\ga_1,\dots,\ga_r)}  \gk^\iij_1(m)=\sum_{\iij\in J(\ga_1,\dots,\ga_r)}  \gs^\iij_1(m)=0$$
when $\ga_t<k_t-e_t-1$ for some $t\in[r]$. Therefore we have:
$$\sum_{\iij\in J(0,\dots,0)}  \gs^\iij_1(m)=
\left\{
  \begin{array}{ll}
   0 , \hfill \hbox{if $e_t<k_t-1$ for some}& t\in[r]; \\
   \displaystyle Q\prod_{t\not\in\gL_3} \Big(-p_t^{k_t-1}(p_t^{e_t+1}-p_t^{k_t}-p_t^{k_t-1})\Big) \prod_{t\in\gL_3}\frac{p_t^{e_t+1}-1}{p_t-1}  , & \hbox{otherwise.}
  \end{array}
\right.
$$
Now we write $\sum'_{\amalg\boldsymbol{\gL}=[r]}$ to mean that in the sum
$\boldsymbol{\gL}=(\gL_1, \gL_2, \gL_3)$ runs
through all partitions of $[r]$ into three parts except for
$(\emptyset,\emptyset,[r])$. Then
{\allowdisplaybreaks
\begin{align*}
    & \sum_{I \subseteq \{1,\dots,r\}} \left(  \prod_{t\in I}\gf(p_t^{k_t} ) \prod_{s\not\in I} p_s^{k_s}
\sum_{p_t|\iij \ \forall t\in I, p_s\nmid \iij \ \forall s\not\in I, 1\le \iij<N} \gk^\iij_1(m) \right) \\
=& \underset{\amalg\boldsymbol{\gL}=[r]} {\sum{}'}\sum_{\bga\vdash\boldsymbol{\gL}}
Q\prod_{t\in\gL_1} \Big(p_t^{k_t}(-p_t^{k_t-1})\Big) \prod_{t\in\gL_2} \Big(\gf(p_t^{k_t})(-p_t^{k_t-1})\Big)
\prod_{t\in\gL_3} \gf(p_t^{k_t})  \Big(\frac{p_t^{e_t+1}-1}{p_t-1}\Big) \\
=& Q\prod_{t=1}^r F_t- Q\prod_{t=1}^r   \gf(p_t^{k_t})\frac{p_t^{e_t+1}-1}{p_t-1} \\
=& Q\prod_{t=1}^r F_t-   \gf(N)\gs_1^0(m)
\end{align*}}\noindent
where  if $e_t<k_t-1$ then
\begin{equation*}
F_t=-\left(\sum_{\ga_t=k_t-e_t-1}^{k_t-1}  \gf(p_t^{k_t}) p_t^{k_t-\ga_t-1} \right)+ \gf(p_t^{k_t})\frac{p_t^{e_t+1}-1}{p_t-1}\\
=0,
\end{equation*}
and if $e_t\ge k_t-1$ then
\begin{align*}
F_t=& -p^{k_t}  p^{k_t-1}-\left(\sum_{\ga_t=1}^{k_t-1}  \gf(p_t^{k_t}) p_t^{k_t-\ga_t-1} \right)+ \gf(p_t^{k_t})\frac{p_t^{e_t+1}-1}{p_t-1}\\
=&-p_t^{k_t-1}(p_t^{e_t+1}-p_t^{k_t}-p_t^{k_t-1}).
\end{align*}\noindent
The theorem now follows at once.
\end{proof}

\begin{cor} \label{cor:AllLevels}
Let $N=\prod_{\iij=1}^r p_\iij^{k_\iij} $ where $p_1,\dots,p_r$ are pairwise distinct prime
factors of $N$. Then we have
\begin{multline*}
 \sum_{\gcd(N,\iij)=1, 1\le \iij<N}  \tg^\iij_2(q)- \gf(N)f_2^0(q) =\\
 \sum_{I \subseteq \{1,\dots,r\}} \left(  \prod_{t\in I}\gf(p_t^{k_t} ) \prod_{s\not\in I} p_s^{k_s}
\sum_{p_t|\iij \ \forall t\in I, p_s\nmid \iij \ \forall s\not\in I, 1\le \iij<N} f^\iij_2(q) \right) .
\end{multline*}
\end{cor}

\begin{eg} \label{eg:PrimePowerCase}
Let $p$ be a prime and the level $N=p^k$. For any $m\in \N$  we have
\begin{equation*}
\sum_{p\nmid \iij, 1\le \iij<N} \tg^\iij_2(q)-\gf(N)f^0_2(q)
=\gf(N)\sum_{p|\iij, 1\le \iij<N}f^\iij_2(q)+N\sum_{p\nmid \iij, 1\le \iij<N} f^\iij_2(q) .
\end{equation*}
\end{eg}

\begin{eg} \label{eg:2PrimePowerCase}
Let $p_1$ and $p_1$ be two distinct primes and $N=p_1^j p_2^k$.
Then we have
\begin{multline*}
 \sum_{p_1\nmid \iij, p_2\nmid \iij, 1\le \iij<N}  \tg^\iij_2(q)-\gf(N)f^0_2(q)
=\gf(N)\sum_{p_1p_2|\iij, 1\le \iij<N}  f^\iij_2(q) \\
  +p_1^j \gf(p_2^k)\sum_{p_1\nmid\ell,p_2|\iij, 1\le \iij<N} f^\iij_2(q)
 + \gf(p_1^j) p_2^k\sum_{p_1|\ell,p_2\nmid\iij, 1\le \iij<N} f^\iij_2(q)
 +N \sum_{p_1\nmid \iij,p_2\nmid \iij,  1\le \iij<N}  f^\iij_2(q).
\end{multline*}
\end{eg}

\section{A linear algebra problem}
In this section, using the standard techniques from linear algebra
and the key result on the multiple divisor functions at level $N$ proved in the
proceeding section we will derive the solvability
of a system of linear equations associated with \eqref{equ:gamStuffleAllN}
and \eqref{equ:gamShuffleAllN} for every positive integer $N$. This completes the proof of
our main result on the level $N$ Eisenstein series given in Theorem~\ref{thm:EisenSeriesDS}.

For every positive integer $N$ we let  $\nu(N)$ be the number of its positive divisors
(including 1 and $N$ itself).
\begin{thm} \label{prop:gams}
For every positive integer $N$ the system
\eqref{equ:gamStuffleAllN} together with \eqref{equ:gamShuffleAllN} has infinitely
many sets of solutions of $\gam_N^{a,b}$ and $\gl_N^{a,b}=\gl_N^{b,a}$ ($0\le a,b<N$)
in terms of $f_2^a(q)$ and $\tg_2^a(q)$ ($0\le a<N$). Moreover one can always choose
\begin{equation}\label{equ:freeVar}
 \{\gam_N^{a,b}: 0\le b<a<N\}\cup \{\gam_N^{N-a,N-a}: 1\le a\le N, a|N\}
\end{equation}
as the $N(N-1)/2+\nu(N)$ free variables.
\end{thm}

Before giving its proof, we first analyze the linear system in Proposition~\ref{prop:gams} using
standard techniques from linear algebra. Let $\bfx_N$ be
a column vector with $(3N^2+N)/2$ components whose transpose is
\begin{align*}
 {}^{t}\bfx_N=(
& \gl_N^{0,0}, \gl_N^{0,1}, \gl_N^{0,2}, \ldots, \gl_N^{N-1,N-1},
\gam_N^{0,1}, \gam_N^{0,2}, \gam_N^{0,3},\ldots, \gam_N^{N-2,N-1},  \\
&\gam_N^{0,0}, \gam_N^{1,1}, \gam_N^{2,2},\ldots, \gam_N^{N-1,N-1},
\gam_N^{1,0}, \gam_N^{2,0},\gam_N^{2,1}, \ldots, \gam_N^{N-1,N-2}).
\end{align*}
Here the rule to list the entries is to use lexicographic order for  $\gl_N^{a,b}$ ($0\le a\le b<N$), then
$\gam_N^{a,b}$ ($0\le a<b<N$), then $\gam_N^{a,a}$ ($0\le a<N$), and finally  $\gam_N^{a,b}$ ($0\le b<a<N$).
Then we can rewrite the system \eqref{equ:gamStuffleAllN} together with \eqref{equ:gamShuffleAllN}
as follows:

\bigskip

$\quad  \ \left\{  \phantom{\frac{\raisebox{.7cm}{A}}{B} } \right. $
\vskip-2.5cm
\begin{alignat}{2}
 \gam_N^{a+b,b}+\gam_N^{a+b,a}-2\gl_N^{a,b}= & \phantom{A} f^{a}_{2}+f^{b}_{2},  & \quad \forall 0\le a\le b<N,
  \qquad  \label{equ:newsystem1} \tag{LS$_1^{a,b}$} \\
 \gam_N^{a,b}+\gam_N^{b,a}-\gam_N^{a+b,a}-\gam_N^{a+b,b } = & -f^{a}_{2}-f^{b}_{2}, & \quad \forall 0\le a<b<N,
    \qquad  \label{equ:newsystem2} \tag{LS$_2^{a,b}$}  \\
 \gam_N^{a,a}-\gam_N^{2a,a} =  &\phantom{A} \tg^{a}_2-f^{a}_{2}, \     & \quad \forall 1 \le a <N, \
 \quad  \qquad   \label{equ:newsystem3}\tag{LS$_3^a$}
\end{alignat}\noindent
where the last two families of the equations are obtained by taking 
the difference of \eqref{equ:gamStuffleAllN} and \eqref{equ:gamShuffleAllN}.
We then can express this system by a single matrix equation
\begin{equation}\label{equ:LinearAlgSystem}
A_N \bfx_N=\bfb_N
\end{equation}
for some matrix $A_N$ of size $(N^2+N-1)\times (3N^2+N)/2$ and a column vector
$\bfb_N$ of length  $N^2+N$ whose entries are given
in terms of $f_2^a$'s and $\tg_2^a$'s only. Notice that since (LS$_3^0$) 
is trivial the row size is decreased from $N^2+N$ by one. To prove the 
proposition one thing we need to show is that every row vector in the 
left null space $\calN(A_N)$  of $A_N$ annihilates $\bfb_N$.

\begin{eg} When $N=1$ we get the equation
$$[-2 \quad 2]
\left[\begin{matrix}
\gl_1^{0,0} \\
\gam_1^{0,0}
\end{matrix}\right]
=2f_2^0.$$
Clearly $\calN(A_1)=\emptyset$ and we may choose
$\gam_1^{0,0}$ arbitrarily and then set $\gl_1^{0,0}=\gam_1^{0,0}=f^0_2$.
\end{eg}

\begin{eg} When $N=2$ we get the equation
\begin{equation}\label{equ:A2}
A_2 \bfx_2=\left[\begin{matrix}
 -2 & 0 & 0 & 0 & 2 & 0 & 0 \\
 0 & -2 & 0 & 0 & 0 & 1 & 1 \\
 0 & 0 & -2 & 2 & 0 & 0 & 0 \\
 0 & 0 & 0 & -1 & 0 & 1 & 0\\
 0 & 0 & 0 & 1 & 0 & -1 & 0
\end{matrix}\right]
\left[\begin{matrix}
\gl_2^{0,0}\\
\gl_2^{0,1}\\
\gl_2^{1,1}\\
\gam_2^{0,1} \\
\gam_2^{0,0} \\
\gam_2^{1,1} \\
\gam_2^{1,0}
\end{matrix}\right]
=\left[\begin{matrix}
f_2^0 \\
f_2^0+f_2^1 \\
f_2^1 \\
\tg_2^1- f_2^1\\
-f_2^0-f_2^1
\end{matrix}\right]=\bfb_2.
\end{equation}
Then $\calN(A_2)$ is spanned by the vector $\bfn_2=(0,0,0,0,1,1)$. We see that
$$ \bfn_2\cdot \bfb_2=f_2^0+2f_2^1-\tg_2^1=0$$
which follows from Example~\ref{eg:PrimePowerCase} by taking $p=2$ and $k=1$ there.
This implies that the system \eqref{equ:A2} has infinitely many solutions.
Setting  $\gam_2^{a,b}=0$ for $1\ge a\ge b\ge 0$ (in fact,
one may choose them arbitrarily as they are free variables)
we only need to solve the system
$$A'_2 \bfx_2=\left[\begin{matrix}
 -2 & 0 & 0 & 0 \\
 0 & -2 & 0 & 0 \\
 0 & 0 & -2 & 2 \\
 0 & 0 & 0 & 1
\end{matrix}\right]
\left[\begin{matrix}
\gl_2^{0,0}\\
\gl_2^{0,1}\\
\gl_2^{1,1}\\
\gam_2^{0,1}
\end{matrix}\right]
=\left[\begin{matrix}
f_2^0 \\
f_2^0+f_2^1 \\
f_2^1 \\
-f_2^0-f_2^1
\end{matrix}\right]=\bfb'_2.$$
Here we obtain $A'_2$ from $A_2$ by removing the penultimate row of $A_2$
(which is equivalent to removing the equation  $\gam_N^{1,1}-\gam_N^{0,1}=\tg^{1}_{2}- f^{1}_2$),
and then removing the last three columns (which is equivalent to setting $\gam_2^{a,b}=0$ for $1\ge a\ge b\ge 0$).
Correspondingly, we obtain $\bfb'_2$
by removing the  penultimate entry $\tg_2^1- f_2^1$  of $\bfb_2$.
Clearly, this new system has a unique solution which also
gives a solution of the original system:
$$\gl_2^{0,0}=-f_2^0,\quad \gl_2^{0,1}=\frac{f_2^1-\tg_2^1}2, \quad \gl_2^{1,1}=-\tg_2^1, \quad \gam_2^{0,1}=2 \gl_2^{0,1},
\quad  \gam_2^{a,b}=0\quad \forall 1\ge a\ge b\ge 0.$$
We can also check that \eqref{equ:gamN=2} with \eqref{equ:glN=2} provides
another set of solution of \eqref{equ:A2}.
\end{eg}

\begin{rem}\label{rem:KT2012}
In an email Kaneko and Tasaka pointed out to us that \cite[(19)]{KanekoTa2013}
should be corrected as follows:
$$\alpha_1=\bar{g}_0^\bfo(q), \quad  \alpha_2=-\alpha_1,\quad
\alpha_3=2\bar{g}_0^\bfo(q)+\bar{g}_0^\bfe(q).
$$
Together with their choice $\gl_2^{0,1}=\gl_2^{1,0}=\gl_2^{1,1}=0$
given in \cite[Theorem 3]{KanekoTa2013} we find the following solution to \eqref{equ:A2}:
\begin{equation*}
\gam_2^{0,0}=f_2^0, \quad
\gam_2^{0,1}=f_2^1, \quad \gam_2^{1,0}=-f_2^1, \quad \gam_2^{1,1}=\tg^1,  \quad
\gl_2^{0,0}=\gl_2^{0,1}=\gl_2^{1,1}=0.
\end{equation*}
since we have the following correspondence between their notation and ours:
\begin{alignat*}{10}
\ga_1 \, & \longleftrightarrow \, & 2\gam_2^{0,1},\quad
&\ga_2 \, & \longleftrightarrow \, & 2\gam_2^{1,0},\quad
&\ga_3 \, &\longleftrightarrow \, & 2\gam_2^{1,1},\\
 \bar g_0^{\bfe}\,  &\longleftrightarrow \, & 2f_2^0,\ \quad
& \bar g_0^{\bfo} \, & \longleftrightarrow \, & 2f_2^1,\ \quad
& g_2^1\,  &\longleftrightarrow \, &\tg_2^1.  \quad
\end{alignat*}
\end{rem}

\begin{eg}
Similarly, when $N=3$ we see that the $\calN(A_3)$ is spanned by the vector
$\bfn_3=(0, 0, 0, 0, 0, 0, 1, 1, 1, 1, 1)$ and
$$\bfn_3 \cdot \bfb_3=g_2^1+g_2^2-2f_2^{0}-3f_2^1-3f_2^2=0$$
which follows from Example~\ref{eg:PrimePowerCase} by taking $p=3$ and $k=1$ there.
Further we can obtain $A'_3$ from the $11\times 15$ matrix $A_3$ by
removing the row of $A_3$ corresponding the equations $\gam_N^{2a,a}-\gam_N^{a,a}= f^{a}_{2}-\tg^{a}_2$ for $a=2$
and then removing the 10th and the last 4 columns
(which is equivalent to setting $\gam_2^{a,b}=0$ for all $2\ge a\ge b\ge 0$ with $(a,b)\ne(1,1)$).
In this way we find the following solution:
\begin{align*}
& \gl_3^{0,0}=-f_2^0, \quad \gl_3^{0,1}=\frac{\tg_2^1-f_2^{0}-2f_2^1}2,
 \quad \gl_3^{1,1}=-f_2^1, \quad \gl_3^{1,2}=\frac{\tg_2^2}2,\\
& \gl_3^{0,2}=-\frac{f_2^0+f_2^2}2,
\quad \gl_3^{2,2}=2\gl_3^{1,2}, \quad \gam_3^{0,1}=\tg_2^1-f_2^0-2f_2^1,\quad \gam_3^{0,2}=-f_2^0-f_2^2, \\
&  \gam_3^{1,2}=f_2^2-\tg_2^2, \quad \gam_3^{1,1}=\tg_2^1-f_2^1, \quad \gam_3^{a,b}=0 \ \forall a\ge b, (a,b)\ne(1,1).
\end{align*}\noindent
Of course, this solution is not unique. For instance, we checked that \eqref{equ:gamN=3}
with \eqref{equ:glN=3} provides another set of solution
of the system $A_3 \bfx_3=\bfb_3$.
\end{eg}

For other levels $N\le 80$ we carried out similar computations by Maple
and verified that $\calN(A_N)$ always annihilates $\bfb_N$ using Corollary~\ref{cor:AllLevels}.
To prove the general case, we need two results concerning dimensions.

\begin{prop} \label{lem:leftNullSpace}
The dimension of the left null space of $\calN(A_N)$ satisfies
$$\dim \calN(A_N)\ge \nu(N)-1.$$
Moreover, for every vector $\bfn\in \calN(A_N)$ we have $\bfn\cdot \bfb_N=0$.
\end{prop}

\begin{proof}
Throughout this proof we will drop the subscript $N$.
For each divisor $d$ of $N$, if $d<N$ we can obtain a vector $\bfn_N(d)\in \calN(A_N)$
by using the following combination of the families of equations in \eqref{equ:newsystem2}
and \eqref{equ:newsystem3}:

\ \phantom{a} \hskip1cm adding \eqref{equ:newsystem2} with $\gcd(b,N)=d$ and $a=0$,

\ \phantom{a} \hskip1cm adding \eqref{equ:newsystem2} with $\gcd(a,b,N)=d$ for all $1\le a<b<N$, and

\ \phantom{a} \hskip1cm adding  \eqref{equ:newsystem3} with $\gcd(a,N)=d$.

\noindent
This gives rise to the vector $\bfn_N(d)$ whose entries are either 0 or 1 and whose leading 1
occurs at the position corresponding to the variable $\gam^{0,d}$.
We will show that $\bfn_N(d)\in \calN(A_N)$ and $\bfn_N(d)\cdot \bfb_N=0$ which
is equivalent to the fact that ${\rm LHS}=0$ and ${\rm RHS}=0$, respectively, where
\begin{equation}\label{equ:LHS}
\text{LHS}= \sum_{\substack{ \gcd(a,N)=d\\ 1\le a<N}}(\gam^{0,a}-\gam^{2a,a})+
 \sum_{\substack{\gcd(a,b,N)=d\\ 1\le a<b<N}}(\gam^{a,b}+\gam^{b,a}-\gam^{a+b,a}-\gam^{a+b,b } )
\end{equation}
and
\begin{equation*}
\text{RHS}= \sum_{\substack{ \gcd(a,N)=d\\ 1\le a<N}}(\tg_2^a-f_2^a)
    -\sum_{\substack{ \gcd(a,N)=d\\ 1\le a<N}}(f_2^0+f_2^a)
    -\sum_{\substack{\gcd(a,b,N)=d\\ 1\le a<b<N}}(f_2^a+f_2^b ).
\end{equation*}
Also notice that the vectors in $\{\bfn_N(d): d|N, d<N\}$ are linearly independent
since the leading 1 appearing in $\bfn_N(d)$ corresponds to $\gam^{0,d}$.
This implies that $\dim\calN(A_N)\ge \nu(N)-1.$

By considering $(N/d)$-th roots of unity (i.e.\ reducing to level $N/d$) we may
assume without loss of generality that $d=1$ and we simply write $\bfn$ for $\bfn_N(1)$.
To show ${\rm LHS}=0$ we break the
two sums in \eqref{equ:LHS} into the following parts:
$$\begin{array}{ll}
\text{($A_a$): $\gam^{0,a}$ from 1st sum,}   &
\text{($B_a$): $-\gam^{2a,a}$, 1st sum,}  \\
\text{($C_1^{a,b}$): $\gam^{a,b}=\gam^{N+a,b}$, 2nd sum, } &
\text{($C_2^{a,b}$): $\gam^{b,a}$, 2nd sum, }  \\
\text{($D_1^{a,b}$): $-\gam^{a+b,a}$, 2nd sum, }  &
\text{($D_2^{a,b}$): $-\gam^{a+b,b}$, 2nd sum, }\ \\
\text{($E_1^a$): $-\gam^{0,a}$, 2nd sum where $b=N-a$}   & \text{(so $a<N/2$),  }  \ \\
\text{($E_2^b$): $-\gam^{0,b}$, 2nd sum where $a=N-b$}   &  \text{(so $b>N/2$).  }
\end{array}$$
Then we obtain the cancelations as follows:
\begin{align*}
   \sum_{\gcd(a,N)=1}  A_a+\sum_{\gcd(a,N)=1,a<N/2} E_1^a+\sum_{\gcd(b,N)=1,b>N/2} E_2^b =0\\
   \sum_{\gcd(a,N)=1}  A_a+\sum_{\gcd(a,N)=1,a<N/2} E_1^a+\sum_{\gcd(b,N)=1,b>N/2} E_2^b =0\\
   \sum_{\substack{\gcd(a,b,N)=1\\ a<b}}(C_1^{a,b}+C_2^{a,b})=
   \sum_{\substack{\gcd(a,b,N)=1\\ 2N>a>b>0}} \gam^{a,b}=CC_1+CC_2+CC_3
\end{align*}\noindent
where
\begin{align*}
&   CC_1=\sum_{\substack{\gcd(a,b,N)=1\\ 2N>a=2b>0}} \gam^{a,b} =-\sum_{\gcd(a,N)=1} B_a,\\
&    CC_2=\sum_{\substack{\gcd(a,b,N)=1\\ 2N>a>2b>0}} \gam^{a,b}  =
   \sum_{\substack{\gcd(a,b,N)=1\\ 2N>a>2b>0}} \gam^{(a-b)+b,b}=-\sum_{\substack{\gcd(a,b,N)=d\\ a<b}} D_1^{a,b} \\
&     CC_3=\sum_{\substack{\gcd(a,b,N)=1\\ 2N>2b>a>b>0}} \gam^{a,b}  =
    \sum_{\substack{\gcd(a,b,N)=1\\ 2N>a>2b>0}} \gam^{(a-b)+b,b}=-\sum_{\substack{\gcd(a,b,N)=d\\ a<b}} D_2^{a,b}.
\end{align*}\noindent
These implies ${\rm LHS}=0$ which shows that $\bfn\in \calN(A_N)$.

Now we turn to RHS. Since $N=1$ case is trivial we now assume
$N=\prod_{\iij=1}^r p_\iij^{k_\iij}>1$ where $p_1,\dots,p_r$ are pairwise distinct prime
factors of $N$. Recall that
\begin{equation}\label{equ:proofRHS}
\text{RHS}= \sum_{\substack{ \gcd(a,N)=d\\ 1\le a<N}}(\tg_2^a-f_2^a)
    -\sum_{\substack{ \gcd(a,N)=d\\ 1\le a<N}}(f_2^0+f_2^a)
    -\sum_{\substack{\gcd(a,b,N)=d\\ 1\le a<b<N}}(f_2^a+f_2^b ).
\end{equation}
We want to show that the above expression is exactly equal to the difference
of the two sides in Corollary~\ref{cor:AllLevels}, which is therefore 0.
Clearly the coefficients of
$\tg_2^a$ (=1) and $f_2^0$ ($=\gf(N)$) are correct.

Let $\gcd(c,N)=1$. Now we count how many times
$f_2^c$ can appear. Notice that if $a\ne c$ then we have
$\gcd(a,c,N)=1$ and either $a<c$ or $a>c$. Thus the last sum in
\eqref{equ:proofRHS} contributes
$N-1$ copies of $f_2^a$. Combining this with the one copy from
the sum $\sum_{\gcd(a,N)=1}(f_2^0+f_2^a)$ we see that the coefficient
of $f_2^a$ is exactly $N$.

Now we consider $f_2^c$ with $\gcd(c,N)>1$. Without loss of generality we
assume that there is $1\le t<r$ such that $p_\iij|c$ for all $\iij\le t$ and
$p_j\nmid c$ for all $t<j\le r$. Then only the last sum in \eqref{equ:proofRHS}
has nontrivial contributions. In fact, for $1\le a<N$ we see that $p_\iij\nmid a$
($\iij=1,\dots,t$) if and only if $\gcd(a,c,N)=1$. But obviously the number of
such $a$ is given by
$$\prod_{\iij=1}^t \gf(p_\iij^{k_\iij} ) \prod_{j=t+1}^r  p_j^{k_j}$$
which agrees with Corollary~\ref{cor:AllLevels}. This implies that ${\rm RHS}=0$
which shows that $\bfn\cdot \bfb_N=0$.

We have completed the proof of the lemma.
\end{proof}

\begin{prop} \label{lem:rankAN}
The rank of matrix $A_N$  satisfies
$$\rank(A_N)\ge N^2+N-\nu(N).$$
Moreover, one may choose the free variables as in \eqref{equ:freeVar}
when solving the linear system $\eqref{equ:gamStuffleAllN}+\eqref{equ:gamShuffleAllN}$
(or, equivalently, the linear system $\eqref{equ:newsystem1}+\eqref{equ:newsystem2}+\eqref{equ:newsystem3}$).
\end{prop}
\begin{proof}
We will drop the subscript $N$ again for $\gl$'s and $\gam$'s.
We will prove the lemma by producing the following pivot variables
in $\bfx_N$:
\begin{equation}\label{equ:varSet}
\calS=\{ \gl^{a,b}:0\le a\le b<N\} \cup
\{ \gam^{a,b}: 0\le a<b<N\} \cup\{ \gam^{a,a}:1\le a<N,(N-a)\nmid N\}.
\end{equation}
Easy computation shows that the $|\calS|=N^2+N-\nu(N)$
which yields the lemma immediately.

To streamline our proof we start with some ad hoc terminology.
Suppose we have a linear system of variables $x_1,\dots,x_r$ 
(in this particular order). Then an equation produced from this 
system by the elementary operations (namely, multiplying
an equation by a scalar and adding or subtracting two equations)
is called a \emph{pivotal equation}
of variable $x_i$ if $x_i$ appears in the equation while none of 
$x_1,\dots,x_{i-1}$ does. In our situation, our variables $\gl$'s 
and $\gam$'s are ordered as in the vector $\bfx_N$. And clearly 
\eqref{equ:newsystem1} provides the  pivotal equations of $\gl$'s.

We now turn to $\gam$'s. We shall produce their pivotal equations 
by the following steps. We write $\gam^{a,b}=\cdots$ to mean that 
the right hand side does not involve any variables from $\calS$. 
In particular, we may omit all $\gam^{a,b}$ with $a>b$.
So by \eqref{equ:newsystem2} we get the pivotal equation
\begin{equation}\label{equ:step1}
 \gam^{a,b}=\cdots \quad\text{for $1\le a<b<N$ and $a+b<p$.}
\end{equation}
By \eqref{equ:newsystem3} we have
\begin{equation}\label{equ:step2}
\gam^{a,a}=\cdots \quad\text{for $1\le a<N/2$.}
\end{equation}
To derive pivotal equation
$\gam^{a,a}=\cdots$ for $N/2<a<N$ with $(N-a)\nmid N$ (thus $a\le N-2$)
we first notice that for such $a$ there must be some positive integer $k\le N-2 $ such that
\begin{equation}\label{equ:aheight}
0\le \frac{(k-1)N}{k}<a<\frac{kN}{k+1}\le \frac{(N-2)N}{N-1}<N-1.
\end{equation}
Here $a$ is bounded with strict inequality because if $a=(k-1)N/k$ for some $k>1$
then $N-a=N/k$ is a divisor of $N$ which is impossible by our assumption.
We say such an $a$ satisfying \eqref{equ:aheight} has height $h(a)=k$.
If $h(a)=1$ then we are in the case of \eqref{equ:step2}.
If $h(a)=2$ then $3a<2N$, i.e. $(2a-N)+a<N$, so using \eqref{equ:newsystem3} we have
\begin{equation*}
\gam^{a,a}=\gam^{2a-N,a}+\cdots=\cdots
\end{equation*}
by \eqref{equ:step1} since $2a-N<a$. If $h(a)=3$ then by
applying \eqref{equ:newsystem3} followed by (LS$_2^{2a-N,a}$) we get
\begin{equation*}
\gam^{a,a}=\gam^{2a-N,a}+\cdots=\gam^{3a-2N,2a-N}+\gam^{3a-2N,a}+\cdots=\cdots
\end{equation*}
by \eqref{equ:step1} again since now $4a<3N$ (and hence $5a<4N$). Repeating 
this process for $a$ at higher levels we obtain the following binary tree

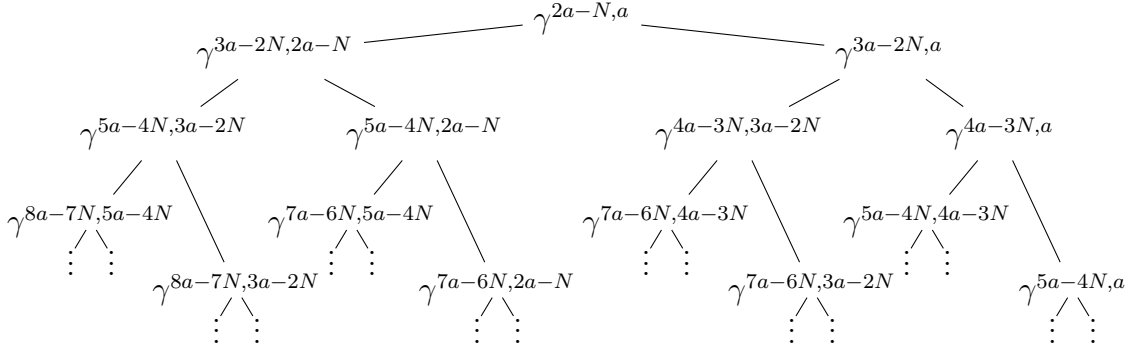
\begin{center}
\begin{tikzpicture}[scale=0.9]
\node (gs) at (0,3.5) {$\gam^{2a-N,a}$};
\node (A1) at (-4.5375,3) {$\gam^{3a-2N,2a-N}$};
\node (A2) at (4.5375,3)  {$\gam^{3a-2N,a}$};
\node (B1) at (-6.175,1.8) {$\gam^{5a-4N,3a-2N}$};
\node (B2) at (-2.325,1.8) {$\gam^{5a-4N,2a-N}$};
\node (B3) at (2.325,1.8) {$\gam^{4a-3N,3a-2N}$};
\node (B4) at (6.175,1.8) {$\gam^{4a-3N,a}$};

\node (C1) at (-7.25,0.5) {$\gam^{8a-7N,5a-4N}$};
\node (C2) at (-5.1,-0.5) {$\gam^{8a-7N,3a-2N}$};

\node (C3) at (-3.4,0.5) {$\gam^{7a-6N,5a-4N}$};
\node (C4) at (-1.25,-0.5) {$\gam^{7a-6N,2a-N}$};

\node (C5) at (1.25,0.5) {$\gam^{7a-6N,4a-3N}$};
\node (C6) at (3.4,-0.5)  {$\gam^{7a-6N,3a-2N}$};

\node (C7) at (5.1,0.5) {$\gam^{5a-4N,4a-3N}$};
\node (C8) at (7.25,-0.5) {$\gam^{5a-4N,a}$};

\node (D1) at (-7.55,0) {$\vdots$};
\node (D2) at (-6.95,0) {$\vdots$};
\node (D3) at (-5.4,-1) {$\vdots$};
\node (D4) at (-4.8,-1) {$\vdots$};

\node (D5) at (-3.7,0) {$\vdots$};
\node (D6) at (-3.1,0) {$\vdots$};
\node (D7) at (-1.55,-1) {$\vdots$};
\node (D8) at (-0.95,-1) {$\vdots$};

\node (D9) at (0.95,0){$\vdots$};
\node (D10) at (1.55,0){$\vdots$};
\node (D11) at (3.1,-1){$\vdots$};
\node (D12) at (3.7,-1){$\vdots$};

\node (D13) at (4.8,0) {$\vdots$};
\node (D14) at (5.4,0) {$\vdots$};
\node (D15) at (6.95,-1) {$\vdots$};
\node (D16) at (7.55,-1) {$\vdots$};

\draw (gs) to (A1) ;
\draw (gs) to (A2) ;
\draw (A1) to (B1) ;
\draw (A1) to (B2) ;
\draw (A2) to (B3) ;
\draw (A2) to (B4) ;
\draw (B1) to (C1) ;
\draw (B1) to (C2) ;
\draw (B2) to (C3) ;
\draw (B2) to (C4) ;
\draw (B3) to (C5) ;
\draw (B3) to (C6) ;
\draw (B4) to (C7) ;
\draw (B4) to (C8) ;
\draw (C1) to (D1);
\draw (C1) to (D2);
\draw (C2) to (D3);
\draw (C2) to (D4);
\draw (C3) to (D5);
\draw (C3) to (D6);
\draw (C4) to (D7);
\draw (C4) to (D8);
\draw (C5) to (D9);
\draw (C5) to (D10);
\draw (C6) to (D11);
\draw (C6) to (D12);
\draw (C7) to (D13);
\draw (C7) to (D14);
\draw (C8) to (D15);
\draw (C8) to (D16);
\end{tikzpicture}
\end{center}
In general, this tree is constructed by the following rules: if a node 
$\gam^{i,j}$ exists and $i+j>N$ (we call $i+j$ the weight of $\gam^{i,j}$) 
then it produces two descendants: $\gam^{i+j-N,i}$ and $\gam^{i+j-N,j}$; 
if $i+j<N$ then it does not have any descendant and therefore becomes a 
terminal node. It is an easy matter
by the induction to show the following properties of this tree:
\begin{itemize}
  \item Every descendant has smaller weight than its parent. So the tree is finite.

  \item Every node has the form $\gam^{ma-(m-1)N,na-(n-1)N}$ for some integers $m>n\ge 1$.

  \item The weight of every node $\gam^{ma-(m-1)N,na-(n-1)N}$ satisfies $(m+n)a-(m+n-2)N\ne N$. Otherwise $N-a=N/(m+n)$
  is a divisor of $N$ which is impossible.

  \item Every node $\gam^{ma-(m-1)N,na-(n-1)N}$ satisfies $ma-(m-1)N\ne 0$. Otherwise $N-a=N/m$
  is a divisor of $N$ which is impossible.

  \item Every node $\gam^{ma-(m-1)N,na-(n-1)N}$ satisfies $ma-(m-1)N<na-(n-1)N$.
\end{itemize}
Hence, every terminal node $\gam^{i,j}$ of the tree satisfies $1\le i<j$ and $i+j<N$
so it can be canceled by using \eqref{equ:step1}.

To summarize the above, we have produced the pivotal equations for the following

\begin{itemize}
  \item[\upshape{(i)}]  $\gam^{a,b}=\cdots$ for $0\le a<b<N$ and $a+b<N$.

  \item[\upshape{(ii)}]   $\gam^{a,a}=\cdots$ for $1\le a<N$ with $(N-a)\nmid N$.

  Then we may proceed as follows:

  \item[\upshape{(iii)}]   $\gam^{0,a}=\gam^{a,a}+\cdots=\cdots$ for 
  $1\le a<N$  by \eqref{equ:newsystem3} and using (ii).

  \item[\upshape{(iv)}]  $\gam^{a,b}=\gam^{0,a}+\gam^{0,b}+\cdots=\cdots$  
  for $1\le a< b<N$ and $a+b=N$ by \eqref{equ:newsystem2} and (iii).

  \item[\upshape{(v)}]  $\gam^{a,b}=\gam^{a+b-N,a}+\gam^{a+b-N,b}+\cdots=\cdots$ for $0\le a<b<N$ and $a+b>N$ by \eqref{equ:newsystem2}
  and by using the induction on the weight $a+b$ since the weights 
  on the right are strictly smaller than $a+b$.
\end{itemize}
We now have finished the proof of the lemma.
\end{proof}

Finally, Proposition~\ref{prop:gams} follows from Lemma~\ref{lem:leftNullSpace}
and Lemma~\ref{lem:rankAN} immediately since
$$\rank(A_N)+\dim \calN(A_N)=N^2+N-1.$$

\medskip
\noindent{\bf Acknowledgement.} This work was started when both
authors were visiting the Morningside Center of Mathematics in Beijing in 2013.
JZ also would like to thank the Max-Planck Institute for Mathematics,
the Kavli Institute for Theoretical Physics China, and the
National Taiwan University for their hospitality where part of 
this work was done. This paper has benefited
greatly from the anonymous referee's detailed
comments and suggestions including Proposition~\ref{prop:HurwitzZeta}.
HY is partially supported by the 2013 summer research grant from
York College of Pennsylvania. JZ is partially supported by NSF grant DMS1162116.

\end{document}